\documentclass[12pt]{amsart}
\usepackage{multirow}
\usepackage[margin=1.1in]{geometry}
\usepackage{amssymb,amsthm}
\usepackage{amssymb}
\usepackage{epsfig}
\usepackage{hyperref}
\usepackage{pstricks}
\usepackage{pst-plot}
\newtheorem{lem}{Lemma}
\newtheorem{lemma}[lem]{Lemma}
\newtheorem{prop}[lem]{Proposition}
\newtheorem{cor}[lem]{Corollary}
\newtheorem{theorem}[lem]{Theorem}
\numberwithin{lem}{section}

\theoremstyle{remark}
\newtheorem{remark}[lem]{Remark}
\newtheorem{example}[lem]{Example}

\numberwithin{equation}{section}

\def\A{{\mathcal A}}
\def\Z{{\mathbb Z}}
\def\L{{\mathcal L}}

\def\P{{\mathcal P}}
\def\x{{\bf x}}
\def\F{{\bf F}}
\def\G{{\bf G}}
\def\I{{\mathcal I}}
\def\FF{{\mathcal F}}
\def\head{{\rm head}}
\def\hF{{\hat F}}
\def\hP{{\hat P}}
\def\hQ{{\hat Q}}
\def\hR{{\hat R}}
\def\Frac{{\rm Frac}}

\def\Q{{\mathbb Q}}

\def\s{{\mathcal S}}
\def\tB{{\tilde B}}

\def\W{{\mathcal W}}

\def\yy{{\mathbf y}}

\begin{document}
\author{Thomas Lam}\address
 {Department of Mathematics\\ University of Michigan\\ Ann Arbor\\ MI 48109 USA.}
 \date{\today}
 \email{tfylam@umich.edu}
 \urladdr{http://www.math.lsa.umich.edu/\~{ }tfylam}
 \thanks{T.L. was supported by NSF grant DMS-0901111 and by a Sloan Fellowship.}
  \author{Pavlo Pylyavskyy}\address
{Department of Mathematics\\ University of Minnesota\\ Minneapolis\\ MN 55414 USA.}
 \email{ppylyavs@umn.edu}
 \urladdr{http://sites.google.com/site/pylyavskyy/}
 \thanks{P.P. was supported by NSF grant DMS-0757165.}
\title{Laurent phenomenon algebras}
\maketitle
\begin{abstract}
We generalize Fomin and Zelevinsky's cluster algebras by allowing exchange polynomials to be arbitrary irreducible polynomials, rather than binomials.
\end{abstract}

\section{Introduction}
In their paper \cite{CA1} Fomin and Zelevinsky introduced a remarkable algebraic object called cluster algebras. The original motivation was to provide a combinatorial model 
for studying total positivity and Lusztig's canonical bases for semisimple Lie groups.  
It was quickly realized however that cluster algebras are rather ubiquitous in mathematics, appearing for example in the
representation theory of quivers and finite-dimensional algebras, Poisson geometry, Teichm\"uller theory, integrable systems, and the study of Donaldson-Thomas invariants.
 
The core idea of cluster algebras is that the generators of a commutative algebra, called {\it cluster variables}, are grouped into sets called {\it {clusters}}. A {\it {seed}} consists of a cluster together with a polynomial, called the {\it {exchange polynomial}}, associated with each cluster variable. The exchange polynomial must be a polynomial in the other variables of this cluster, and is always 
a binomial. One can then apply a {\it {mutation procedure}} to a variable in a cluster, exchanging it for a different variable
according to the following rule:
$$\text{old variable} \times \text{new variable} = \text{exchange binomial}.$$ 
The exchange polynomials are also mutated, producing a mutated seed from the old seed.  One key remarkable property of such systems then is the {\it {Laurent phenomenon}},
which says that any cluster variable is a Laurent polynomial when written as a rational function in any other cluster. 

From the onset of the theory it was known that the Laurent phenomenon holds in a more general setting, where the exchange polynomials are not necessarily binomials: Fomin and Zelevinsky \cite{FZLP} established the Laurent phenomenon for a number of families of combinatorial recurrences, including the Somos sequences, the cube recurrence, and the Gale-Robinson sequence.
However, the work of \cite{FZLP} depended on already knowing the global pattern of exchange polynomials, the Laurentness with respect to which one is trying to establish.  What \cite{FZLP} does {\it {not}} provide is a rule on 
how to derive the global exchange pattern from knowing only the local one in an initial seed, which is what is achieved for cluster algebras.

In this work, we propose a method to propagate arbitrary (irreducible) exchange polynomials. We then prove that the Laurent phenomenon always holds, and we call our new algebras {\it Laurent phenomenon algebras}, or {\it LP algebras}.  
The new paradigm of mutation that we offer is as follows:
$$\text{old variable} \times \text{new variable} = \text{exchange Laurent polynomial}.$$
Here the Laurent polynomial on the right hand side is equal to the exchange polynomial of the variable divided by a monomial in the rest of the variables in the same cluster. 
The exchange polynomials of a seed determine its exchange Laurent polynomials.

Let us list some features of cluster algebras which extend, or conjecturally extend to LP algebras:
\begin{enumerate}
 \item The Laurent phenomenon (Theorem \ref{T:Laurent}) holds for LP algebras (cf. \cite{CA1}).
 \item There is a rich theory of finite types, and the associated cluster complexes appear to be polytopal complexes with rich combinatorics; see Sections \ref{sec:r2} and \ref{sec:ex} (cf. \cite{CA2,CFZ,FZ2}).  In \cite{LP2} we study LP algebras with a linear seed, and in particular we show that the number of finite types of LP algebras grows exponentially with rank.
 \item The cluster monomials appear to be linearly independent, and for finite type LP algebras appear to form linear bases (cf. \cite{CK}).
 \item For a suitable initial seed, the cluster variables appear to be Laurent polynomials with {\it positive} coefficients; see Section \ref{sec:r2} and \cite{LP2}.
 \item There are interesting examples of LP algebras of {\it {finite mutation type}}; see Section \ref{sec:ex} (cf. \cite{FST}).
\item The coefficients of exchange polynomials satisfy interesting dynamics under mutation; see Section \ref{sec:r2} (cf. \cite{CA4}).
\item Beautiful combinatorial recurrences occur as exchange relations of LP algebras, including the Gale-Robinson sequence and cube recurrence; see Section \ref{sec:ex} (cf. \cite{FZLP,FZ2,Pro}).
 \item LP algebras appear naturally as coordinate rings of Lie groups or certain varieties naturally associated to Lie groups (cf. \cite{CA3}).
\end{enumerate}

Let us elaborate on the last point.  The initial motivating examples of cluster algebras were the coordinate rings of double Bruhat cells of semisimple Lie groups \cite{CA3}.  
In \cite{LPEL} we constructed a family of {\it electrical Lie groups} naturally associated with electrical networks in a disk.  
The positive parts of these electrical Lie groups come with a decomposition into cells analogous to the Bruhat decomposition of the totally positive part of the unipotent subgroup of a semisimple group.  The dynamics of parametrizations of these cells is controlled by {\it electrical LP algebras} in the same way the dynamics of parametrizations of Bruhat cells is controlled by cluster algebras \cite{CA3}.  In the upcoming work \cite{LP3} we shall explain the details. We refer the reader to Section \ref{sec:ex} for an example.

Let us list some differences between cluster algebras and Laurent phenomenon algebras:
\begin{enumerate}
\item
For certain initial seeds, the cluster algebra generated by that seed may not be the same as the LP algebra generated by that seed; in Corollary \ref{C:cluster} we show that this never happens if the cluster algebra has principal coefficients.
 \item In the definition of seed mutation of a LP algebra, a substitution is first made in an exchange polynomial and then a (possibly very interesting) polynomial factor is removed; in the cluster case this factor is always just a monomial.
 \item In a LP algebra it is possible for the exchange polynomial of one cluster variable to depend on another cluster variable, while the reverse is not true; in the cluster case this relation is always symmetric.
 \item In a LP algebra mutation a priori depends on the exchange polynomials of all cluster variables of the seed, including cluster variables which are not being mutated.  The extent to which this dependence is not present is a very interesting question, a special case of which is addressed in \cite[Theorem 2.4]{LP2}.  In the cluster case {\it freezing} a variable by never mutating it is straightforward.
\item In a LP algebra, the {\it {cluster complex}} that describes which variables can belong to the same cluster is not necessarily a flag complex (that is, it is not necessarily given by just pairwise
compatibility), see Remark \ref{rem:flag}; this property is known to hold for cluster algebras arising from surfaces, and is
conjectured for cluster algebras in general \cite{FSTh}. 
\end{enumerate}

The paper is organized as follows. In Section \ref{sec:seeds} we define seeds and seed mutation of LP algebras, and establish their basic properties. In Section \ref{sec:LP} we give the definition of LP algebras. 
In Section \ref{sec:LPCL} we compare cluster algebras with LP algebras, and discuss sufficient conditions for a cluster algebra to be a LP algebra.  In Section \ref{sec:cat} we prove that the
Laurent phenomenon holds for LP algebras. In Section \ref{sec:r2} we give a complete classification of rank two LP algebras of finite type. In Section \ref{sec:ex} we discuss several interesting families of examples of LP algebras, recovering and explaining connections to the work of Chekhov and Shapiro \cite{ChSh}, Hone \cite{Ho}, and Henriques and Speyer \cite{HS}.




\section{Seeds and seed mutation}\label{sec:seeds}

Recall that an element $f \in A$ of a unique factorization domain is {\it irreducible} if it is non-zero, not a unit, and cannot be expressed as the product $f = gh$ of 
two elements $g, h \in A$ which are non-units.  If $f, g \in A$ and $g$ is not a unit, and not zero in $A$ it makes sense to ask for the highest power of $g$ that divides $f$.

\subsection{Seeds}
Much of our notation and terminology imitates that in the theory of cluster algebras \cite{CA1,CA2,CA3,CA4}.

Let $S$ be a {\it coefficient ring} over $\Z$, which we assume to be a unique factorization domain.  For example $S$ could be $\Z$, a polynomial ring over $\Z$, or a Laurent polynomial ring over $\Z$.  Let $n \geq 1$ be a positive integer and write $[n]$ for $\{1,2,\ldots,n\}$.  Let the {\it ambient field} $\FF$ be the rational function field in $n$ independent variables over the field of fractions $\Frac(S)$. 

A {\it seed} in $\FF$ is a pair $(\x,\F)$ where
\begin{itemize}
\item
$\x = \{x_1,x_2,\ldots,x_n\}$ is a transcendence basis for $\FF$ over $\Frac(S)$.
\item
$\F = \{F_1,F_2,\ldots,F_n\}$ is a collection of polynomials in $\P = S[x_1,x_2,\ldots,x_n]$ satisfying:
\begin{enumerate}
\item[(LP1)]
$F_i$ is an irreducible element of $\P$ and is not divisible by any variable $x_j$
\item[(LP2)]
$F_i$ does not involve the variable $x_i$
\end{enumerate}
\end{itemize}
The variables $\{x_1,x_2,\ldots,x_n\}$ are called {\it cluster variables}, and the polynomials $$\{F_1,F_2,\ldots,F_n\}$$ are called {\it exchange polynomials}.  As is usual in the theory of cluster algebras, the set $\{x_1,x_2,\ldots,x_n\}$ will be called a {\it cluster}.  If $t = (\x,\F)$ is a seed, we let $\L = \L(t)$ denote the Laurent polynomial ring $S[x_1^{\pm 1},x_2^{\pm 1},\ldots,x_n^{\pm 1}]$.  If $x$ is a cluster variable, we shall use the notation $F_x$ to denote the exchange polynomial associated to a cluster variable $x$.  This is to be distinguished from the notation $F(y)$ in use later.  We call $n$ the {\it rank} of the seed $(\x,\F)$.

For two polynomials $f,g \in \P$, or more generally two elements $f, g \in \FF$, write $f \propto g$ to mean that $f$ and $g$ differ (multiplicatively) by a unit in $S$.

\begin{remark}
The sets $\{x_1,x_2,\ldots,x_n\}$ and $\{F_1,F_2,\ldots,F_n\}$ are considered to be unordered, but the information of which exchange polynomial corresponds to which cluster variable is given.  When giving an example, we will often list a seed by giving a set of ordered pairs, each pair $(x_i,F_i)$ consisting of a cluster variable and its exchange polynomial.
\end{remark}
\begin{remark}
The mutation dynamics that we shall discuss sometimes behave properly even for seeds $t$ not satisfying the irreducibility condition of (LP1), though in all the examples we have encountered  we can reduce to consider a seed $t'$ which does satisfy (LP1), for example by changing the coefficient ring, or by introducing new coefficients. 
\end{remark}

For each seed $(\x,\F)$, we define a collection $\{\hF_1,\hF_2,\ldots,\hF_n\} \subset \L$ of {\it exchange Laurent polynomials} by the conditions:
\begin{itemize}
\item
$\hF_j = x_1^{a_1}\cdots \widehat{x_j} \cdots x_n^{a_n} F_j$ for some $a_1,\ldots,a_{j-1},a_{j+1},\ldots,a_n \in \Z_{\leq 0}$
\item for $i \neq j$ we have that
\begin{align}\label{hatF}
&\hF_i|_{x_j \leftarrow F_j/x} \text{ lies in } S[x_1^{\pm 1},\ldots,x_{j-1}^{\pm 1},x^{\pm 1},x_{j+1}^{\pm 1},\ldots,x_n^{\pm 1}] \; \mbox{and,}\\
&\mbox{as an element of this ring, is not divisible by $F_j$.} \nonumber
\end{align}
\end{itemize}

The well-definedness of $\hF_i$ follows from the following lemma.
\begin{lemma}
Let $F(x) \in S[x^{\pm 1}]$ be a Laurent polynomial in $x$ with coefficients in a unique factorization domain $S$.  Let $P$ be an irreducible element of $S$.  Then there is a unique integer $m \in \Z$ so that $G(x) = x^m F(x)$ satisfies the following two properties:
\begin{enumerate}
\item
$G(P/x) \in S[x^{\pm 1}]$
\item
$G(P/x)$ is not divisible by $P$
\end{enumerate}
\end{lemma}
\begin{proof}
We may assume that $F$ is a polynomial in $x$.  Then $F(P/x) \in S[x^{\pm 1}]$.  Let $m$ be negative of the maximal power of $P$ that divides $F(P/x)$.  Then clearly $G(x) = x^m F(x)$ satisfies both (1) and (2) and this value of $m$ is unique.  
\end{proof}

\begin{lemma}
The collections $\{F_1,\ldots,F_n\}$ and $\{\hF_1,\ldots,\hF_n\}$ determine each other uniquely.
\end{lemma}
\begin{proof}
The definition of the $\hF_i$-s tells us how to uniquely obtain them from the $F_i$-s. For the reverse direction, we simply drop the denominators of $\hF_i$-s to recover the $F_i$-s. 

More formally, we have that 
\begin{equation}\label{eq:FihFi}
F_i = x_1^{b_1}\cdots \widehat{x_i} \cdots x_n^{b_n} \hF_i
\end{equation} where $b_j \in \Z_{\geq 0}$ are the minimal nonnegative integers such that $F_i$ is a genuine (not Laurent) polynomial. Indeed, by definition the $\hF_i$-s are obtained from the polynomials $F_i$-s by dividing by some nonnegative powers of $x_j$'s. We need to compensate for this division when we go in the other direction. If we chose larger values of $b_j$ in \eqref{eq:FihFi}, then $F_i$ would not be irreducible, contradicting (LP1).  If we chose smaller values of $b_j$ in \eqref{eq:FihFi}, then $F_i$ would not be a genuine polynomial.  Thus the integers $b_j$ in \eqref{eq:FihFi} are uniquely determined.
\end{proof}

Let us pause now and consider an example.
\begin{example} \label{ex:1}
Let $S = \Z$ and $\FF = \Q(a,b,c)$.  Consider the seed
$$t= \{(a, b+1), (b, (a+1)^2+c^2), (c, b^2+b+a^3+a^2)\}.$$
Then $\hF_a=F_a$ since both $F_b$ and $F_c$ depend on $a$, and thus $b$ and $c$ cannot appear in the denominator of $\hF_a$ with a non-zero exponent. Similarly, $\hF_b=F_b$. For the same reason $\hF_c$ does not have any $b$-s in its denominator. 
To compute the exponent of $a$ in the denominator of $\hF_c$ we make the substitution $a \longleftarrow \frac{b+1}{a'}$ in $F_c$, obtaining 
$$b(b+1)+\frac{(b+1)^2(b+1+a')}{(a')^3}.$$ The maximal power of $b+1$ that divides this is $(b+1)^1$, and therefore $\hF_c = F_c/a$.
\end{example}

\begin{lemma}\label{L:subs}
In \eqref{hatF}, the substitution $x_j \leftarrow F_j/x$ can be replaced by $x_j \leftarrow \hF_j/x$ without changing the condition.  
Similarly, in \eqref{hatF} we can test divisibility by $\hF_j$ instead of by $F_j$.
\end{lemma}
\begin{proof}
We begin by noting that $F_j$ and $\hF_j$ do not depend on $x$. 

Consider an arbitrary Laurent polynomial $P \in S[x_1^{\pm 1},\ldots,x_{j-1}^{\pm 1},x_j^{\pm 1},x_{j+1}^{\pm 1},\ldots,x_n^{\pm 1}]$.  If $F_j$ divides $T(x) = P|_{x_j \leftarrow F_j/x}$ in $S[x_1^{\pm 1},\ldots,x_{j-1}^{\pm 1},x^{\pm 1},x_{j+1}^{\pm 1},\ldots,x_n^{\pm 1}]$, then it divides all the coefficients $c_r$ of 
$T(x) = \sum_{r} c_r x^r$ as a Laurent polynomial in $x$.  Let $F_j^p$ be the maximal power of $F_j$ that divides $T(x)$.  Then $p = \min_r p_r$, where $F_j^{p_r}$ is the maximal power of $F_j$ that divides $c_r$.

Now consider $T'(x) = P|_{x_j \leftarrow \hF_j/x} = \sum_{r} c'_r x^r$.  Since $F_j$ and $\hF_j$ differ by (multiplication by) a unit not involving $x$ in $S[x_1^{\pm 1},\ldots,x_{j-1}^{\pm 1},x^{\pm 1},x_{j+1}^{\pm 1},\ldots,x_n^{\pm 1}]$, we have that $c'_r/c_r$ is a unit in $S[x_1^{\pm 1},\ldots,x_{j-1}^{\pm 1},x_{j+1}^{\pm 1},\ldots,x_n^{\pm 1}]$.  Suppose that $F_j^{p'_r}$ is the maximal power of $F_j$ that divides $c'_r$.  Then $p'_r = p_r$.  It follows that $F_j^p$ is the maximal power of $F_j$ that divides $T'(x)$.  This proves the first statement.

The second statement also follows from the fact that $\hF_j$ and $F_j$ differ multiplicatively by a unit in $S[x_1^{\pm 1},\ldots,x_{j-1}^{\pm 1},x^{\pm 1},x_{j+1}^{\pm 1},\ldots,x_n^{\pm 1}]$, and thus the maximal powers of each that divides any coefficient $c_j$ of $T(x)$ (or of $T'(x)$) are the same. 
\end{proof}

\begin{lemma}\label{L:depend}
Suppose that $F_j/\hF_j$ involves $x_i$.  Then $F_i$ does not use the variable $x_j$.
\end{lemma}

\begin{proof}
The fact that $F_j/\hF_j$ involves $x_i$ means that there is a non-trivial power of $F_i$ that divides $F_j|_{x_i \leftarrow F_i/x}$. Indeed, by the definition \eqref{hatF} of $\hF_j$ this power is exactly the power of $x_i$ in $F_j/\hF_j$.

Now by (LP1), $F_j$ is not divisible by $x_i$, and has a non-zero constant term $c = F_j|_{x_i \leftarrow 0}$ when viewed as a polynomial in $x_i$.  Since $F_i$ divides $F_j|_{x_i \leftarrow F_i/x}$, it also divides $c$.  Since $F_j$ does not depend on $x_j$, the constant term $c$ also does not depend on $x_j$, and we conclude that $F_i$ does not depend on $x_j$.
\end{proof}

\begin{example}
In Example \ref{ex:1}, the ratio ${F_c}/{\hF_c} = a$ involves $a$, and indeed we see that $c$ does not appear in $F_a = b+1$.  This agrees with Lemma \ref{L:depend}.
\end{example}

We may think of $\F$ and ${\bf \hF}$ as two different normalizations for the tuple of exchange polynomials.  They are defined up to a monomial product in the $x_i$'s.  
The set $\F$ consists of the unique representatives which are polynomials not divisible by any variable.  
The elements of ${\bf \hF}$ are the unique representatives satisfying \eqref{hatF}.

\begin{lemma} \label{lem:hcrit}
 If $F_i \not = F_j$, the exponent $a_i$ in the definition of $\hF_j$ is maximal such that $\hF_j$ is a Laurent polynomial in $\L(\mu_i(\x,\F))$.
\end{lemma}

\begin{proof}
The largest power of $F_i$ that divides $F_j|_{x_i \leftarrow F_i/x}$ is clearly the largest $a_i$ such that ${F_j}/{x_i^{a_i}}$ remains Laurent polynomial after the substitution $x_i \leftarrow F_i/x$. 
\end{proof}

\subsection{Mutations}
Suppose $i \in [n]$.  Then we say that a tuple $(\x',\F')$ is obtained from a seed $(\x,\F)$ by mutation at $i$, and write $(\x',\F')  = \mu_i(\x,\F)$, if the former can be obtained from the latter by the following 
(non-deterministic) procedure.  

The cluster variables of $\mu_i(\x,\F)$ are given by $x'_i = \hF_i/x_i$ and $x'_j = x_j$ for $j \neq i$.  
The exchange polynomials $F'_j \in \L' := \L(t')$ are obtained from $F_j$ as follows.  First, we define $F'_i := F_i$.  For $j \neq i$ we have two cases.  
If $F_j$ does not depend on $x_i$, then $F_j$ is also an element of $\L'$, and we define $F'_j$ to be any polynomial satisfying $F'_j\propto F_j$, where $F'_j$ is now considered as an element of $\L'$.  

Now suppose $F_j$ does depend on $x_i$.  By Lemma \ref{L:depend}, $x_j$ cannot appear in the denominator $F_i/\hF_i$ of $\hF_i$, so $\hF_i|_{x_j \leftarrow 0}$ is well defined.  We define $G_j$ by
\begin{equation}\label{E:G}
G_j = F_j|_{x_i \leftarrow \frac{\hF_i|_{x_j \leftarrow 0}}{x'_i}}
\end{equation}
Next, we define $H_j$ to be $G_j$ with all common factors (in $S[x_1,\ldots,\hat x_i,\ldots,\hat x_j \ldots,x_n]$) with $\hF_i|_{x_j \leftarrow 0}$ removed.  
Note that this defines $H_j$ only up to a unit in $S$.  Finally, we define 
\begin{equation} \label{E:M}
F'_j = MH_j
\end{equation} 
where $M$ is a Laurent monomial in the $x'_1,x'_2,\ldots,\widehat{x'_j},\ldots, x'_n$ 
with coefficient a unit in $S$, such that $F'_j \in \P' = S[x_1', \ldots, x_n']$, satisfies (LP2), and is not divisible by any variable in $\P'$.  For any $H_j$, it is always possible to pick the monomial $M$ to satisfy these conditions, but in general there are many choices for the coefficient of $M$.  In particular $F'_j$ is defined only up to a unit in $S$.

\begin{example} \label{ex:2}
Consider the seed $$t= \{(a, b+1), (b, (a+1)^2+c^2), (c, b^2+b+a^3+a^2)\}$$ from Example \ref{ex:1}. Recall that $\hF_a=F_a, \hF_b=F_b$ and $\hF_c = F_c/a$. 
Let us see what happens when we mutate at $c$. The variable $c$ changes into
$$d = \frac{\hF_c}{c} = \frac{b^2+b+a^3+a^2}{ac}.$$ 
The exchange polynomial $F_a$ does not change (or changes only by a unit in $S = \Z$) since it does not depend on $c$. To compute the new $F_b$, we make the substitution $$c \longleftarrow \frac{\hF_c|_{b=0}}{d} = \frac{a(a+1)}{d}.$$ The result is 
$(a+1)^2+(\frac{a(a+1)}{d})^2$. Now we need to kill all common factors it has with $a(a+1)$, and change it to an irreducible polynomial by multiplying by a monomial. The first step kills 
the factor $(a+1)^2$, resulting in $1 + (\frac{a}{d})^2$, and the second step turns it into $a^2+d^2$. Thus, the resulting mutated seed can be chosen to be
$$\mu_c(t) = \{(a, b+1), (b, a^2+d^2), (d, b^2+b+a^3+a^2)\}.$$
One can verify from the definition that this is indeed a valid seed. 
\end{example}

We shall now show that if $(\x',\F') = \mu_i(\x,\F)$ is obtained by mutation of $(\x,\F)$ at $i$ then (LP1) is automatically satisfied, so $(\x',\F')$ is also a seed.  It is clear that if $\x$ is a transcendence basis of $\FF$ over $\Frac(S)$, then so is $\x'$.

\begin{lemma}\label{L:variablesinvolved}
Assume we are mutating at $i \in [n]$.  Then $F_j$ depends on $x_i$ if and only if $F'_j$ depends on $x'_i$.
\end{lemma}

\begin{proof}
If $F_j$ does not depend on $x_i$, then $F'_j \propto F_j$ does not depend on $x'_i$.  If $F_j$ depends on $x_i$, then $G_j$ must involve $x'_i$ since $\hF_i\mid_{x_j \leftarrow 0}$ is non-zero by (LP1).  But the remaining operations will not change the fact that $x'_i$ is involved (using also that $F_j$ is not divisible by $x_i$).
\end{proof}

\begin{example}
 Compare the two seeds 
 $$t= \{(a, b+1), (b, (a+1)^2+c^2), (c, b^2+b+a^3+a^2)\}$$
and
 $$\mu_c(t) = \{(a, b+1), (b, a^2+d^2), (d, b^2+b+a^3+a^2)\}$$
 from Example \ref{ex:2}. We see that before the mutation $F_a$ does not depend on $c$, while $F_b$ does. Similarly, after the mutation $F_a$ does not depend on $d$, while $F_b$ does.   This agrees with Lemma \ref{L:variablesinvolved}.
\end{example}

\begin{lemma}\label{L:hFi}
Assume we are mutating at $i \in [n]$.  Then $\hF_i = \hF'_i$.
\end{lemma}

\begin{proof}
By definition we have $F'_i = F_i$, so we need to know that, for each $j$, the same power of $F_j$ divides $F_i\mid_{x_j \leftarrow F_j/x}$ as the power of $F'_j$ divides $F_i\mid_{x_j \leftarrow F'_j/x}$.  If $F_j$ does not depend on $x_i$, then $F_j \propto F'_j$ so this is clear.  
On the other hand, if $F_j$ depends on $x_i$ then $F'_j$ also depends on $x'_i$ by Lemma \ref{L:variablesinvolved}.  But then by Lemma \ref{L:depend} we conclude that both $\hF_i/F_i$ and $\hF'_i/F'_i$ do not involve $x_j$.
Thus $\hF_i/F_i$ and $\hF'_i/F'_i$ has the same power of $x_j$, for any $j$, and the lemma follows. 
\end{proof}

\begin{example} \label{ex:3}
We verify Lemma \ref{L:hFi} in Example \ref{ex:2}.  Let us compute $\hF_d$ in the seed $$\mu_c(t) = \{(a, b+1), (b, a^2+d^2), (d, b^2+b+a^3+a^2)\}.$$ Since $F_b$ depends on $d$, there is no $b$ in the denominator of $\hF_d$. To find the exponent of $a$ in this denominator, make the substitution 
 $a \longleftarrow \frac{b+1}{a'}$ in $F_d$, obtaining 
$$b(b+1)+\frac{(b+1)^2(b+1+a')}{(a')^3}.$$ The maximal power of $b+1$ that divides this is $(b+1)^1$, and therefore $\hF_d = F_d/a = (b^2+b+a^3+a^2)/a$. As we have seen in Example \ref{ex:1}, this coincides with $\hF_c$ in the original seed $t$, agreeing with Lemma \ref{L:hFi}.
\end{example}

\begin{prop}\label{P:irred}
Mutation at $i$ gives a valid seed $$\mu_i(\x,\F) = (\{x'_1,\ldots,x'_n\}, \{F'_1,\ldots,F'_n\}).$$ 
\end{prop}
\begin{proof}  
We need to check the condition (LP1).  By construction, it only remains to show that the $F'_j$ are irreducible in $\P'$.  This is clear if $F_j$ does not involve $x_i$.  Suppose otherwise.  
Then $F'_j$ involves $x'_i$, so it is a non-constant polynomial in $\P'=S[x_1', \ldots, x_n']$, and in particular is not a unit in $\P'$.  (Indeed, the only units in $\P'$ are the units in $S$.) 
Suppose $F'_j = P_1P_2$, for $P_r \in \P'$ non-units.  Let $Z =\hF_i|_{x_j \leftarrow 0}$.  Then
$G_j = A \cdot M^{-1} \cdot P_1P_2,$
where $G_j$ is as in \eqref{E:G}, and $M$ is the Laurent monomial of \eqref{E:M}, and we have that $A \in \P' \cap \P$ and the irreducible factors of $A$ are factors of $Z$.  Since $M$ is a Laurent monomial in $x'_1,x'_2,\ldots,\widehat{x'_j},\ldots, x'_n$ (together with a unit coefficient in $S$), up to a unit in $\L$, we have that $M|_{x'_i \leftarrow \frac{Z}{x_i}}$ is just a power of $Z$.  Now $$
F_j = G_j|_{x'_i \leftarrow \frac{Z}{x_i}},
$$
and $F_j$ is irreducible, so it follows that one of 
$P_r|_{x'_i \leftarrow \frac{Z}{x_i}}$ for $r = 1,2$ is either (i) a unit in $\L$, or (ii) it is a product of factors of $Z$ with a unit in $\L$.  Since $F'_j$ is not divisible by any $x'_k$, it follows that $P_r$ is not a monomial in $\P'$.  Since $Z$ does not involve $x_i$ or $x'_i$, case (ii) is only possible if $P_r$ does not involve $x'_i$, and hence $P_r$ is itself divisible by a factor of $Z$.  This would contradict the definition of $H_j$, so we must be in case (i): $P_r|_{x'_i \leftarrow \frac{Z}{x_i}}$ is a 
unit in $\L$.  If $P_r$ does not involve $x'_i$, then it is also a unit in $\P'$, which is a contradiction.  Finally, if $P_r$ involves $x'_i$, and since it is not 
divisible by $x'_i$ in $\P'$, it is clear that $P_r|_{x'_i \leftarrow \frac{Z}{x_i}}$ cannot
 be a unit. Indeed, this is because $P_r|_{x'_i \leftarrow \frac{Z}{x_i}}$ is not a monomial, as $P_r$ had terms with distinct degrees of $x'_i$, and this property is 
preserved under substitution.
\end{proof}

\begin{prop}\label{P:mutationinverse}
If $(\x',\F')$ is obtained from $(\x,\F)$ by mutation at $i$, then $(\x,\F)$ can be obtained from $(\x',\F')$ by mutation at $i$.
\end{prop}

\begin{proof}
By Lemma \ref{L:hFi}, mutating at $i$ twice we reproduce the same cluster variables $x_1,\ldots,x_n$.  Thus we need only focus on whether we can recover (up to a unit) $F_j$ for $j \neq i$.  If $F_j$ does not involve $x_i$ this follows from the definition.  
Now suppose that $F_j$ does involve the variable $x_i$.  Let $F''_j$ denote the result of mutating $F'_j$ at $i$ (since mutation is not completely deterministic, we are taking $F''_j$ to be any such mutation).  We have
$$
G_j = A \cdot M^{-1} \cdot F'_j
$$ 
as in the proof of Proposition \ref{P:irred}.  Now $F_j = G_j|_{x'_i \leftarrow \frac{Z}{x_i}}$ and $F_j$ is irreducible so $F_j$ must divide $A|_{x'_i \leftarrow \frac{Z}{x_i}}$, $M^{-1}|_{x'_i \leftarrow \frac{Z}{x_i}}$ or $F'_j|_{x'_i \leftarrow \frac{Z}{x_i}}$.  By assumption, $F_j$ involves $x_i$, so it does not divide $A|_{x'_i \leftarrow \frac{Z}{x_i}}$ or $M^{-1}|_{x'_i \leftarrow \frac{Z}{x_i}}$.  
Thus it must divide $F'_j|_{x'_i \leftarrow \frac{Z}{x_i}}$.  It is easy to see then that $F_j$ divides $F''_j$. Indeed, $F''_j$ differs from $F'_j|_{x'_i \leftarrow \frac{Z}{x_i}}$
by a Laurent monomial factor, and by a factor consisting of common divisors with $Z$. Neither one can be divisible by $F_j$, since it is not monomial and depends 
on $x_i$, unlike $Z$.  Irreducibility of $F_j$ and $F''_j$ now implies the statement.
\end{proof}

\begin{example} \label{ex:4}
Consider the seed $$t= \{(a, b+1), (b, (a+1)^2+c^2), (c, b^2+b+a^3+a^2)\}$$ from Example \ref{ex:1} and its mutation $$\mu_c(t) = \{(a, b+1), (b, a^2+d^2), (d, b^2+b+a^3+a^2)\}$$
 from Example \ref{ex:2}. We saw in Example \ref{ex:3} that $\hF_d = F_d/a$. It can be easily seen that $\hF_a=F_a$ and $\hF_b=F_b$ in $\mu_c(t)$, just like in Example \ref{ex:1}. 
Let us see what happens when we mutate the seed $\mu_c(t)$ at $d$. The variable $d$ changes into
$$\frac{\hF_d}{d} = \frac{b^2+b+a^3+a^2}{ad} = c.$$ 
The exchange polynomial $F_a$ does not change (or changes only by a unit in $S = \Z$) since it does not depend on $d$. To compute the new $F_b$, we make the substitution $$d \longleftarrow \frac{\hF_d|_{b=0}}{c} = \frac{a(a+1)}{c}.$$ The result is 
$a^2+(\frac{a(a+1)}{c})^2$. Now we need to kill all common factors it has with $a(a+1)$, and change it to an irreducible polynomial by multiplying by a monomial. The first step kills 
the factor $a^2$, resulting in $1 + (\frac{a+1}{c})^2$, and the second step turns it into $(a+1)^2+c^2$. Thus, the resulting mutated seed can be chosen to be
$$\mu_d(\mu_c(t)) = \{(a, b+1), (b, (a+1)^2+c^2), (c, b^2+b+a^3+a^2)\} = t.$$
This agrees with Proposition \ref{P:mutationinverse}.
\end{example}

\section{Laurent phenomenon algebras} \label{sec:LP}
\subsection{Definition}
Let $S$ be a fixed coefficient ring and $\FF$ denote the ambient fraction field in $n$ indeterminates as in Section \ref{sec:seeds}.  A {\it Laurent phenomenon algebra} $(\A, \{(\x,\F\})$ is a subring of $\A \subset \FF$ together with a distinguished collection of seeds $\{(\x,\F)\} \subset \FF$ belonging to the ambient field $\FF$.  The algebra $\A \subset \FF$ is generated over $S$ by all the variables $\x$ in any of the seeds of $\A$.  The seeds satisfy the condition: for each seed $(\x,\F)$ and $i \in [n]$, we are given a seed $(\x',\F') = \mu_i(\x,\F)$ obtained from $(\x,\F)$ by mutation at $i$. Thus the seeds form the vertices of a $n$-regular graph, where the edges are mutations.  Furthermore, we assume all seeds are connected by mutation.  We shall often write $\A$ to mean the pair $(\A,\{(\x,\F)\})$.  To emphasize that the seeds are part of the data we shall say ``LP algebra $\A$'', and if the seeds are not part of the information, we say ``commutative ring $\A$''.

If $t = (\x,\F)$ is any seed in $\FF$, we shall let $\A(t)$ denote any LP algebra which has $t$ as a seed.  We say that $\A(t)$ is generated (as a LP algebra) by $t$, or has {\it initial seed} $t$.  Since seed mutation is only well-defined up to units, the seeds of $\A(t)$ are not determined by $t$.  However, as we shall see presently, the commutative subring $\A(t) \subset \FF$ is determined by $t$.

\subsection{Equivalence of seeds}
Recall that for two elements $f,g \in \FF$, we write $f \propto g$ to mean that $f$ and $g$ differ (multiplicatively) by a unit in $S$.  We say that two seeds $(\x,\F)$ and $(\x',\F')$ are {\it equivalent} if the following two conditions hold:
\begin{enumerate}
\item
For each $i$ we have $x_i \propto x'_i$, and
\item
For each $i$ we have $F_i \propto F'_i$, where $F_i, F'_i$ are viewed as elements of the ambient field $\FF = \Frac(S[x_1,x_2,\ldots,x_n]) = \Frac(S[x'_1,x'_2,\ldots,x'_n])$.
\end{enumerate}

\begin{lemma}\label{lem:equiv}
Suppose $(\x,\F)$ and $(\x',\F')$ are equivalent seeds.  Let $(\yy,\G)$ and $(\yy',\G')$ be obtained from $(\x,\F)$ and $(\x',\F')$ respectively by mutation at $i$.  Then $(\yy,\G)$ and $(\yy',\G')$ are equivalent seeds.
\end{lemma}

\begin{lemma}\label{lem:seednorm}
Suppose $\A(t)$ and $\A'(t)$ are two LP algebras generated by a fixed seed $t$.  Then each seed of $\A(t)$ is equivalent to some seed of $\A'(t)$ and conversely.  In particular, as subrings of $\FF$, the two commutative rings $\A(t)$ and $\A'(t)$ are identical.
\end{lemma}

\begin{example} \label{ex:n1}
 Let $S = \mathbb Z[C,C^{-1}]$ and consider two LP algebras $\A$ and $\A'$ generated by the same seed $$t = \{(a, f+C), (f, a+C)\}.$$ The first LP algebra $\A$ has cluster variables $$a, \; b, \; d=\frac{b+1}{a}, \; e=\frac{bC+a+C}{ab}, \; f=\frac{a+C}{b},$$
 and seeds
 $$\{(a, b+1), (b, a+C)\}, \{(b, Cd+1), (d, b+1)\}, \{(d, e+1), (e, Cd+1)\},$$ $$\{(e, f+C), (f, e+1)\}, \{(a, f+C), (f, a+C)\}.$$
 The second LP algebra $\A'$ has cluster variables $$a, \; b, \; d'=\frac{b+1}{aC}, \; e=\frac{bC+a+C}{ab}, \; f=\frac{a+C}{b},$$
 and seeds
 $$\{(a, (b+1)/C), (b, a+C)\},\{(b, C^2d'+1), (d', (b+1)/C)\},\{(d', (e+1)/C), (e, C^2d'+1)\},$$ $$\{(e, f+C), (f, (e+1)/C)\},\{(a, f+C), (f, a+C)\}.$$
 Then the seeds split into pairs of equivalent ones in the obvious way. For example, the seeds $\{(d, e+1), (e, Cd+1)\}$ and $\{(d', (e+1)/C), (e, C^2d'+1)\}$ are equivalent since $d \propto d'$ and also $e+1 \propto (e+1)/C$ and $Cd+1 \propto C^2d'+1$. It is also easy to see that the rings generated by 
 $a,b,d,e,f$ and $a,b,d',e,f$ over $S = \mathbb Z[C,C^{-1}]$ coincide, in agreement with Lemma \ref{lem:seednorm}. 
\end{example}

\subsection{Normalization}
Let $\A$ be a Laurent phenomenon algebra.  We will say that $\A$ is {\it normalized} if whenever two seeds $t_1$, $t_2$ are equivalent, we have that $t_1 = t_2$.  Suppose $\A'$ is another LP algebra with the same ambient field as $\A$.  Then we will say $\A$ is the normalization of $\A'$ if $\A$ is normalized, and there is a surjective map $p:t' \mapsto t$ sending seeds of $\A'$ to seeds of $\A$ such that
\begin{enumerate}
\item
for each seed $t'$ of $\A'$, we have that $p(t')$ and $t'$ are equivalent, and
\item
for each seed $t'$ of $\A'$ and each $i \in [n]$ we have that $p(\mu_i(t')) = \mu_i(p(t'))$.
\end{enumerate}
By Lemma \ref{lem:equiv}, we see that (1) and the fact that $\A$ is normalized implies (2).  Our usage of ``normalization'' is different from, but related to, the usual usage in cluster algebras \cite{CA1}.

The following result follows from Lemma \ref{lem:seednorm}.
\begin{lemma}\label{lem:normal}
Suppose $\A$ and $\A'$ are two LP algebras both generated by a fixed seed $t$.  If $\A'$ is normalized, then it is the normalization of $\A$.
\end{lemma}

\begin{example}
Both of the LP algebras in Example \ref{ex:n1} are normalized, and thus each of them is a normalization of the other.  One could get a non-normalized LP algebra $\A''$ with ten seeds by taking the multiset union of all the seeds in $\A$ and in $\A'$.  The exchange graph (see Section \ref{sec:exchangegraph}) of $\A''$ is a 10-cycle that goes once through each of the ten seeds: imagine gluing the two 5-cycle exchange graphs of $\A$ and $\A''$ by cutting one of the edges incident to the initial seed $t$. 
\end{example}

\subsection{Finite type and finite mutation type}\label{sec:finitetype}
Suppose $\A$ is a LP algebra which is normalized.  Then we say that $\A$ is of {\it finite type} if it has finitely many seeds.  If $\A$ is not necessarily normalized, we say that $\A$ is of finite type if it has a normalization $\A'$ of finite type.  This condition implies that $\A$ has finitely many equivalence classes of seeds, and the converse holds in rank two (Corollary \ref{cor:rank2}) but is not clear in general. 

Call two seeds $t$ and $t'$ {\it {similar}} if there exists a seed $t''$ equivalent to $t'$ such that $t''$ can be obtained from $t$ by renaming the cluster variables (and substituting this renaming into the exchange polynomials).  In particular, equivalent seeds are similar.  Let us say that an LP algebra is {\it {of finite mutation type}} if it has finitely many similarity classes of seeds. In the case of cluster algebras there is a beautiful theory of cluster algebras of
 finite mutation type, see \cite{FSTh, FST}.  See Sections \ref{sec:GaleRobinson} and \ref{sec:finitemutation} for examples of interesting similarity classes of seeds, and of interesting LP algebras of finite mutation type.

\subsection{Freezing}
Let $\A$ be a rank $n$ Laurent phenomenon algebra, and $(\x,\F)$ a chosen seed in $\A$.  Let $i \in [n]$.  Let $S' = S[x_i^{\pm 1}]$.  Let $\{(\x^{(1)},\F^{(1)}), (\x^{(2)},\F^{(2)}),\ldots\}$ be the subset of seeds that can be obtained from $(\x,\F)$ by mutation at the indices $j \in [n] \setminus i$.  In particular, each seed $(\x^{(k)},\F^{(k)})$ has $x^{(k)}_i = x_i$.

For each seed $(\x^{(k)},\F^{(k)})$ we produce a rank $n-1$ seed $(\x'^{(k)},\F'^{(k)})$ as follows: we remove $(x_i,F_i)$, and we replace $F^{(k)}_j$ for $j \neq i$ by $F'^{(k)}_j := F^{(k)}_j/x_i^d$, where the power $x_i^d$ is the same as that in $\hF^{(k)}_j$.  We also have $x'^{(k)}_j = x^{(k)}_j$ for $j \neq i$.
Let us now consider the collection of rank $n-1$ seeds $\{(\x'^{(k)},\F'^{(k)})\}$.  
The polynomials $F'^{(k)}_i$ are now considered elements of $S'[x_1,\ldots,x_{i-1},x_{i+1},\ldots,x_n]$.  We claim that the seeds $(\x'^{(k)},\F'^{(k)})$ satisfy (LP1) and (LP2). Indeed, $F^{(k)}_j/x_i^d$ is still irreducible in $S'[x_1,\ldots,x_{i-1},x_{i+1},\ldots,x_n]$ since any factorization in $S'[x_1,\ldots,x_{i-1},x_{i+1},\ldots,x_n]$ would, after clearing denominators, give a factorization of $F^{(k)}_j$ in $S[x_1,\ldots,x_n]$.  Also, $F^{(k)}_j/x_i^d$ will not be divisible by $x_k$ for $k \neq i$, and will not depend on $x_j$ (because the same holds for $F^{(k)}_j$).  Let $\A' \subset \FF =\Frac(S'[x_1,\ldots,x_{i-1},x_{i+1},\ldots,x_n])$ be the subalgebra generated by all the variables $\x'$.

\begin{example}
 Let $\A$ be an LP algebra over $S = \mathbb Z$ with the initial seed $$t = \{(a, b+1), (b, a+c), (c,b+1)\}.$$ We mutate at the two variables distinct from $c$ several times to obtain four more seeds: 
 $$\{(b, c^2d+1), (c,b+1), (d, b+1)\}, \{(c,e+1), (d, e+1), (e, c^2d+1)\},$$ $$\{(c,e+1), (e, f+c), (f, e+1)\}, \text{ and }  \{(a, f+c), (c,f+a), (f, a+c)\}.$$  
 
 Let $\A'$ be the LP algebra over $S = \mathbb Z[c,c^{-1}]$ obtained from $\A$ by freezing $c$ in $t$. The seeds of $\A'$ obtained from the above seeds of $\A$ are 
 $$\{(a, (b+1)/c), (b, a+c)\},\{(b, c^2d+1), (d, (b+1)/c)\},\{(d, (e+1)/c), (e, c^2d+1)\},$$ $$\{(e, f+c), (f, (e+1)/c)\},\{(a, f+c), (f, a+c)\},$$ where $$a,b,d=\frac{b+1}{ac},e=\frac{bc+a+c}{ab},f=\frac{a+c}{b}$$ are the variables in those seeds, coinciding with the variables of the original 
 LP algebra $\A$. Here we applied the freezing procedure as described above. For example, in the first seed we have $F_a = (b+1)/c$ since $c$ appears with exponent $1$ in the denominator $F_a/\hF_a$ in $t \in \A$.
The reader can verify that after the identification $C \mapsto c$, the LP algebra $\A'$ from Example \ref{ex:n1} can be identified with the LP algebra $\A'$ in the current example. 
\end{example}

\begin{prop}
The algebra $\A'$, together with the seeds $\{(\x',\F')\}$ are a Laurent phenomenon algebra.
\end{prop}
\begin{proof}
We have already explained that each $(\x',\F')$ is a legitimate LP seed.  We show that if two seeds $t=(\x^{(r)},\F^{(r)})$ and $\mu_\ell(t) = (\x^{(s)},\F^{(s)})$ of $\A$ are related by mutation at $\ell$, then so are $t' = (\x'^{(r)},\F'^{(r)})$ and $\mu_\ell(t)' = (\x'^{(s)},\F'^{(s)})$.  By \eqref{hatF} and the definition of $\F'^{(k)}_j$, we have that $\hF'^{(k)}_j = \hF^{(k)}_j$ for all $k$ and $j \neq i$.  It follows that the equation $x^{(r)}_\ell x^{(s)}_\ell = \hF^{(r)}_\ell$ implies that $x'^{(r)}_\ell x'^{(s)}_\ell = \hF'^{(r)}_\ell$.  
Thus the cluster variables in $\A'$ mutate correctly.
Now we check that the exchange polynomials in $\A'$ mutate correctly.  The equality $\hF'^{(r)}_j = \hF^{(r)}_j$ and the equality $x'^{(r)}_j = x^{(r)}_j$ for all $j \neq i$ implies that we perform the same substitution in \eqref{E:G} to calculate $\mu_\ell$ for the seeds $t$ and $t'$.  The rest of the calculation of the mutation is also the same, and the only difference is that for the seed $t'$, we started with $F'^{(r)}_j = F^{(r)}_j/x_i^d$.  But $x_i$ is a unit in the coefficient ring $S'$ of $\A'$, and mutations of exchange polynomials are defined up to units, so we indeed have $\mu_\ell(t)' = \mu_\ell(t')$.
\end{proof}

\subsection{Cluster complex and exchange graph}\label{sec:exchangegraph}
The {\it cluster complex} of a LP algebra is the simplicial complex with base set equal to the set of cluster variables, and faces corresponding to collections of cluster variables that lie in the same cluster.  The {\it exchange graph} of a LP algebra $\A$ is the graph with vertex set equal to the set of seeds of $\A$, and edges given by mutations.

\begin{remark}\label{rem:flag}
The cluster complex of a LP algebra is not always a flag complex: clusters are not determined by pairwise compatibility.  Take the seed $$t = \{(x_1,P),(x_2,P),(x_3,Q),(x_4,Q)\}$$ where $P,Q \in S$ are irreducible and not proportional.  The normalized LP algebra generated by $t$ has 9 seeds and 6 cluster variables.  Every pair of cluster variables appears together in some cluster.

The corresponding property is conjectured to hold for cluster algebras \cite{FSTh}.
\end{remark}

\section{Comparison with cluster algebras of geometric type} \label{sec:LPCL}
We compare our notion of seeds and seed mutation with those in the theory of cluster algebras.  We will restrict ourselves to cluster algebras of geometric type. 

In this subsection we will take an integer $m \geq n$ and set $S = \Z[x^{\pm 1}_{n+1},x^{\pm 1}_{n+2},\ldots, x^{\pm 1}_m]$.  The variables $x_{n+1},\ldots,x_m$ are called {\it frozen variables}.  A {\it cluster algebra seed} in $\FF$ is a pair $(\x,\tB)$ where
\begin{enumerate}
\item
$\x = \{x_1,x_2,\ldots,x_n\}$ is a transcendence basis for $\FF$ over $\Frac(S)$.
\item
$\tB = (b_{ij})$ is a $m \times n$ {\it exchange matrix} such that the top $n \times n$ submatrix $B$ of $\tB$ is a skew-symmetrizable integer matrix: that is, there exists a ($n \times n$) diagonal matrix $D$ with positive diagonal entries such that the matrix $DB$ is skew-symmetric.
\end{enumerate}
To a cluster algebra seed $(\x,\tB)$ we associate exchange polynomials $\{F_1,\ldots,F_n\}$ defined by
\begin{equation}\label{E:Fcluster}
F_j = \prod_{b_{ij} >0} x_i^{b_{ij}} + \prod_{b_{ij} < 0} x_i^{-b_{ij}}.
\end{equation}
These exchange polynomials are always binomials.  Recall that a vector $v \in \Z^m$ is called {\it primitive} if it is non-zero, and the greatest-common-divisor of the entries of $v$ is equal to $1$.

\begin{lemma}
Suppose the column $(b_{1j},b_{2j},\ldots,b_{mj})$ is a primitive integer vector in $\Z^m$.  Then $F_j$ is irreducible in $S[x_1,\ldots,x_n] = \Z[x_1,x_2,\ldots,x_m]$.
\end{lemma}
\begin{proof}
Our proof will show that $F_j$ is irreducible even with complex coefficients.  The Newton polytope $N(p)$ of a polynomial $p(x_1,\ldots,x_m) \in \Z[x_1,\ldots,x_m]$ is the convex hull of the vectors $(a_1,a_2,\ldots,a_m)$ for all monomials $x_1^{a_1}\cdots x_m^{a_m}$ that appear in $p$.  It is well-known that we have $N(pq) = N(p) + N(q)$ where addition here is the Minkowski sum.  

The Newton polytope $N(F_j)$ of a binomial is a line segment.  If $F_j$ can be factorized non-trivially, then $N(F_j)$ must be the Minkowski-sum of two 
lattice polytopes which are not points.  (In fact, these polytopes must be line segments  parallel to $N(F_j)$.)  In particular, this would imply that $N(F_j)$ 
contained an interior lattice point.  
But this would in turn imply that there is an integer $d > 1$, namely one plus the nuber of such internal points, which divides all the coordinates of the endpoints of $N(F_j)$.  The result follows.
\end{proof}

We say that $\tB$ is {\it primitive} (or $(\x,\tB)$ is {\it primitive}) if the columns of $\tB$ are primitive integer vectors.

\begin{prop}\label{P:CAseed}
A primitive cluster algebra seed $(\x,\tB)$ gives rise to a Laurent phenomenon algebra seed $(\x,\F)$.
\end{prop}

Cluster algebra seed mutation is given as follows.  Let $i \in [n]$.  Then the mutation $\mu_i(\x,\tB)$ of $(\x,\tB)$ at $i$ is given by $(\x',\tB')$ where $x'_i = F_i/x_i$ and $x'_j = x_j$ for $j \neq i$.  Here $F_i$ is given by \eqref{E:Fcluster} and note that the formula uses $F_i$ and not $\hF_i$!  The new exchange matrix is given by
\begin{equation}\label{E:clustermutation}
b'_{kj} = \begin{cases}
-b_{kj} & \mbox{if $k =i$ or $j = i$;} \\
b_{kj} + \dfrac{|b_{ki}|b_{ij}+b_{ki}|b_{ij}|}{2} & \mbox{otherwise.}
\end{cases}
\end{equation}
Given a cluster algebra seed $t = (\x,\tB)$, the {\it cluster algebra} $\A_{CA}(t)$ with initial seed $t$ is the collection of all seeds $t'$ obtained by successive mutation from $t$, together with the subring of $\FF$ generated over $S$ by all cluster variables $\x'$ in any of these seeds.  
The property that the exchange matrix is primitive is preserved under mutation of exchange matrices.
\begin{lemma}\label{L:prim}
Suppose $\tB$ is primitive.  Then so is $\tB' = \mu_i(\tB)$.
\end{lemma}
\begin{proof}
Suppose $d$ divides $b'_{kj}$ for all $k$ and some fixed $j$.  Then by \eqref{E:clustermutation}, $d$ divides $b_{ij}$, and so it must divide  $\dfrac{|b_{ki}|b_{ij}+b_{ki}|b_{ij}|}{2}$ for any $k$.  It follows that $d$ divides $b_{kj}$ for all $k$.  Thus if $\tB$ is primitive, so is $\tB'$.
\end{proof}
Assume now that we are in the situation of Proposition \ref{P:CAseed}.  Let $(\x,\tB)$ be any cluster algebra seed of $\A_{CA}(t_0)$ and assume that the exchange polynomials give a legitimate LP algebra seed $(\x,\F)$.   Let the cluster algebra seed nutation of $(\x,\F)$ at $i$ be $(\x',\tB')$.
Let the Laurent phenomenon seed mutation of $(\x,\F)$ at $i$ be $(\x'',\F'')$.  We want to compare $(\x'',\F'')$ with $(\x',\tB')$.  The new cluster variable $x''_i$ in the Laurent phenomenon seed mutation is given by $\hF_i/x_i$ instead of $F_i/x_i$, so we have the equality $x''_i = x'_i$ if and only if $\hF_i = F_i$.

When do we have $F_i = \hF_i$ for a LP algebra seed arising from a cluster algebra seed?  If $x_i$ occurs in $F_j$ then we know that $x_j$ does not appear in $F_i/\hF_i$ by Lemma \ref{L:depend}.  If $x_i$ does not occur in $F_j$, then $x_j$ also does not occur in $F_i$ by the skew-symmetrizability of $B$.  But then $F_i|_{x_j \leftarrow F_j/x} = F_i$  is divisible by $F_j$ only if $F_i \propto F_j$, 
since $F_i$ and $F_j$ are both irreducible.   In fact, for a cluster algebra seed, we have $F_i \propto F_j$ only if $F_i = F_j$.  This suggests we look at the ``coprime" condition of cluster algebras.

Recall from \cite{CA3} that a cluster algebra seed $(\x,\tB)$ is called {\it coprime} if the exchange binomials $F_i$ are coprime in $S[x_1,x_2,\ldots,x_n]$.  If $(\x,\tB)$ satisfies the condition of Proposition \ref{P:CAseed}, then coprimality is equivalent to the condition that $F_i \neq F_j$ for $i \neq j$, and thus implies that $F_i = \hF_i$ for all $i$.  This suggests the following result.

\begin{prop}\label{P:cluster}
Suppose $t_0 = (\x_0,\tB_0)$ is a primitive cluster algebra seed where $\tB_0$ is a full rank matrix. Then the cluster algebra $\A_{CA}(t)$ of geometric type generated by $(\x_0,\tB_0)$ is a Laurent phenomenon algebra, and for every seed in $\A_{CA}(t)$, cluster algebra seed mutation agrees with LP algebra seed mutation.
\end{prop}
\begin{proof}
By \cite[Proposition 1.8]{CA3}, all seeds mutation equivalent to $(\x_0,\tB_0)$ are coprime.  By Lemma \ref{L:prim}, all these seeds are also primitive.  Thus every seed of $\A_{CA}(t_0)$ is a LP algebra seed which in addition satisfies $\hF_i = F_i$ for all $i \in [n]$. 

Let $(\x,\tB)$ be any cluster algebra seed of $\A_{CA}(t_0)$ and let $(\x,\F)$ denote the corresponding LP algebra seed.  Let the cluster algebra seed obtained from mutation at $i$ be $(\x',\tB')$ and let the Laurent phenomenon seed mutation of $(\x,\F)$ at $i$ be $(\x'',\F'')$.  Since $\hF_i = F_i$, we have $x'_i = x''_i$.

We now check that $F'_j = F''_j$.  First, suppose that $F_j$ does not involve $x_i$.  This happens if and only if $b_{ij} = 0$ and directly from \eqref{E:clustermutation} we have that the $j$-th column of $\tB'$ is the same as the $j$-th column of $\tB$.  It follows that in this case we have $F''_j = F_j = F'_j$.  

Now suppose that $F_j$ does involve $x_i$ and so $b_{ij} \neq 0$.  Suppose $F_j = A+B$ where $A$ is the monomial involving $x_i$.  We now calculate $F''_j$ directly from the definitions in Section \ref{sec:seeds}.  By the skew-symmetrizability condition of $\tB$, we have that $b_{ji} \neq 0$.  Thus $F_i$ involves $x_j$, and so $F_i|_{x_j \leftarrow 0}$ is actually a monomial (rather than a binomial).  Then $G_j = A' + B$ as defined in \eqref{E:G} is the sum of two monomials.  Now let us consider the occurrences of $x_k$ for $k \neq i$ in $G_j$.  We calculate that as long as $b_{ki}$ and $b_{ij}$ have the same sign then $A'/A$ has a factor of $x_k^{b_{ki}b_{ij}}$.  If $b_{kj}$ has the same sign as $b_{ij}$ (or $b_{kj} = 0$) then all powers of $x_k$ in $G_j$ occur in the same monomial.  Otherwise if $b_{kj} \neq 0$ and has opposite sign to $b_{ij}$, then powers of $x_k$ occur in both monomials $A'$ and $B$ of $G_j$.  To obtain $F''_j$ any common factors of $x_k$ are factored out.  A case-by-case computation shows that this is exactly what happens in \eqref{E:clustermutation}, giving the equality $F''_j = F'_j$. Thus $(\x'',\F'')$ is the Laurent phenomenon seed associated to $(\x',\tB')$.  
\end{proof}

The conditions of Proposition \ref{P:cluster} holds for all cluster algebras which have an initial seed with {\it principal coefficients}.  A seed $(\x,\tB)$ has principal coefficients if the matrix $\tB$ is $2n \times n$, and the bottom $n \times n$ submatrix is the identity matrix.  As shown in \cite{CA4}, ``one can think of principal coefficients as a crucial special case providing control over cluster algebras with arbitrary coefficients''.

\begin{cor}\label{C:cluster}
Every cluster algebra with principal coefficients is a Laurent phenomenon algebra.
\end{cor}

\begin{remark}
The full rank and primitive conditions on the exchange matrix $\tB$ can be thought of as certain non-degeneracy conditions on the cluster algebra which have appeared in a number of places in the literature.  For example, constructions by Geiss, Leclerc, and Schroer \cite{GLS} of cluster algebras that are not unique factorization domains fail these conditions.
\end{remark}

\begin{example}
Let us finish with an example of an LP algebra and a cluster algebra which have the same initial seed but are different.  Working with $S = \Z$ and $\FF = \Q(a,b,c)$, consider the following initial seed:
$$t = \{(a, 1+b), (b, a+c), (c, 1+b)\}.$$
(Note that $F_a = F_c$.)  Then there are four more variables in this LP algebra, given by 
$$d = \frac{1+b}{ac}, \;\; e=\frac{a+c}{b}, \;\; e = \frac{a+c+bc}{ab}, \;\; f= \frac{a+c+ab}{bc},$$
and the cluster complex consists of the faces $abc$, $ace$, $cef$, $aeg$, $abd$, $bcd$, $cfd$, $efd$, $egd$, $agd$.
The two other kinds of clusters that appear are 
$$\{(a, e+b), (e, a+c), (c, e+b)\} \;\;\; \text{and} \;\;\; \{(d, 1+b), (b, 1+c^2d), (c, 1+b)\}.$$
On the other hand, the cluster algebra this seed produces is a type $A_3$ cluster algebras with a total of $9$ variables and $14$ clusters.

\end{example}

\section{The caterpillar lemma and Laurent phenomenon} \label{sec:cat}
In this section we establish the namesake property of Laurent phenomenon algebras:

\begin{theorem}\label{T:Laurent}

Let $\A$ be a Laurent phenomenon algebra and $t = (\x,\F)$ be a seed of $\A$.  Then every cluster variable of $\A$ belongs to the Laurent polynomial ring $\L(t) = S[x_1^{\pm 1},\ldots,x_n^{\pm 1}]$.
\end{theorem}

For LP algebras of rank $n \leq 1$ the result is trivial, so we assume $n \geq 2$ from now on.  Our proof follows the same strategy as Fomin and Zelevinsky's work \cite{CA1,FZLP}.
We prove an analogue of Fomin and Zelevinsky's Caterpillar Lemma.  Let $t_0 = t$ contain cluster variables $x,y$, and let $t_1,t_2,t_3$ be the seeds obtained by mutating first at $x$, to get $z$, then at $y$ to get $u$, and finally at $z$ to get $v$, as in the following diagram:
$$
\overset{x,y}{\bullet} \overset{\hP}{\text{---------}}\overset{z,y}{\bullet} \overset{\hQ}{\text{---------}} \overset{z,u}{\bullet} \overset{\hR}{\text{---------}} \overset{v,u}{\bullet}
$$
Here $\hP,\hQ,\hR$ are the exchange Laurent polynomials of the respective mutations, so $xz = \hP$, $yu = \hQ$ and $zv = \hR$.  We shall think of the Laurent polynomials $\hP,\hQ,\hR$ as polynomials in one special variable: $\hP=\hP(y)$, $\hQ = \hQ(z)$ and $\hR = \hR(u)$.  

Let $\L =\L(t_0)$ denote the Laurent polynomial ring for the original cluster containing $x$ and $y$.  In the following, $\gcd$ is always taken inside $\L$.  The greatest common divisor is defined up to a unit, so saying that $\gcd(a,b) = 1$, is the same as saying that the only elements that divide both $a$ and $b$ are units.
\begin{lemma} \label{lem:u}
We have \begin{itemize}
\item
$u\in \L$,
\item
$\gcd(z,u) = 1$.
\end{itemize}
\end{lemma}
\begin{proof}
To show that $u \in \L$, it suffices to show that $\hQ(z) \in \L$.  
But by \eqref{hatF} and Lemma \ref{L:subs}, $\hQ(z)|_{z \leftarrow \hP/x}$ lies in $\L$.  Thus the claim follows from the equality $z = \hP/x$ in $\FF$.

Now, $x$ and $y$ are units in $\L$ and $u = \hQ/y$ and $z = \hP/x$ so
$\gcd(z,u) = \gcd(\hP,\hQ) = \gcd(P,\hQ)$.
Again by \eqref{hatF} and Lemma \ref{L:subs}, $\hQ(x) = \hQ(z)|_{z \leftarrow \hP/x}$ is not divisible by $P$ in $\L$.  Since $P$ is irreducible in $\L$, it follows that $\gcd(P,\hQ) = 1$.
\end{proof}

Recall that $f \propto g$ means that $f$ and $g$ differ multiplicatively by a unit in $S$.  

\begin{lemma} \label{L:allequal}
Suppose that considered as elements of the ambient field $\FF$, we have that $P \propto Q \propto R$ and hence the polynomials do not depend on $y,z,u$ respectively.  Then we have
$$
z = \frac{P}{xyM}, \qquad u \propto x, \qquad \mbox{and } v \propto y,
$$ 
where $M$ is a monomial not involving $x,y,z,v,u$.
\end{lemma}
\begin{proof}
By definition, the exchange polynomials for $x$ and $y$ in $t_0$ are $P$ and $Q$ respectively.  Using the definition \eqref{hatF}, we have $\hP = \dfrac{P}{yM}$, where $M$ is a monomial in the other (not $x$ or $y$) cluster variables of $t_0$.  This gives the formula for $z$.  

Similarly, $\hQ = \dfrac{Q}{zM'}$, giving $u = \dfrac{Q}{yzM'} \propto \dfrac{xM}{M'}$.  From the definition \eqref{hatF} and the assumption $P \propto Q$, we see that the cluster variables $w$ that occur in $M$ are also exactly the ones occurring in $M'$, with the same degree. 
So we have $M = M'$, and $u \propto x$.  The argument for $v \propto y$ is the same.\end{proof}

\begin{lemma} \label{lem:v}
We have
\begin{itemize}
\item
$v \in \L$;
\item
$\gcd(z,v) = 1$.
\end{itemize}
\end{lemma}
\begin{proof}
We have $v = \hR(u)/z$.  
Since $xz = \hP(y)$, we have that $z/P(y)$ is a unit in $\L$, and thus $z$ is irreducible in $\L$ by (LP1).

Case 1: Suppose that $R(u)$ does not depend on $u$.  Then by Lemma \ref{L:variablesinvolved}, by the definition of mutation of exchange polynomials, $R \propto P$ and $P(y)$ does not depend on $y$.  Now $\hR = R \cdot M(u)$, 
where $M(u)$ is a Laurent monomial depending on $u$ (and other cluster variables in $t_2$) and not on $z$.  The power of $u$ that appears in $M(u)$ is equal to $-1$ if $Q(z)$ divides $R$, and equal to $0$ otherwise (we use (LP1) that $R$ and $Q(z)$ are irreducible).  Since $R$ does not depend on $z$, the former occurs if and only if $Q \propto R$, using (LP3).  By Lemma \ref{L:allequal}, $v \propto y$ is a unit in $\L$, so both of the claims follow.  

Thus we may assume that $\hR = R \cdot M$ where $M$ is a Laurent monomial not involving $u$.  Similarly, we may assume that $\hP = P \cdot M'$ where $M'$ is a Laurent monomial not involving $y$.  We calculate
$$
v = \frac{\hR}{z} \propto \frac{P \cdot M}{(P \cdot M')/x}
$$
giving that $v$ is a unit in $\L$, and again both of the claims follow.

\smallskip
Case 2: Suppose that $R(u)$ depends on $u$.  Then by Lemma \ref{L:depend}, $\hQ(z) = Q(z) \cdot M$ for a monomial $M$ not depending on $z$.  
Suppose that $\hR(u) = R(u) \cdot u^{-p} \cdot M'$, for a monomial $M'$ not depending on $u$, and $p \geq 0$.  

Case 2a: Suppose that $\hQ(z)$ depends on $z$.  Then $p = 0$ and $\hR/R$ is a unit in $\L$.  We have
$$
\frac{R(u)}{z} = \frac{R\left(\frac{\hQ(z)}{y}\right)}{z}= \frac{R\left(\frac{\hQ(z)}{y}\right) - R\left(\frac{\hQ(0)}{y}\right)}{z} + \frac{R\left(\frac{\hQ(0)}{y}\right)}{z} 
$$
Since $R(u)$ mutates to $P(y)$, we know that
\begin{equation}\label{E:PMA}
\frac{R\left(\frac{\hQ(0)}{y}\right)}{z} = \frac{P(y) \cdot M'' \cdot A}{z}\end{equation}
where $A$ is the product of some factors of $\hQ(0)$ which can be chosen to be polynomial, and $M'' = M''(y)$ is a Laurent monomial in $y$, and the other variables (that is, $M''$ does not involve $z$ or $u$).  Note that $\hQ(0) \in \L$ and thus $A \in \L$.  As $z = \hP(y)/x$ and $\hP(y)/P(y)$ is a unit in $\L$, it follows that $\frac{R(\frac{Q(0)}{y})}{z} \in \L$.  Also, $f(z) = R\left(\frac{\hQ(z)}{y}\right) - R\left(\frac{\hQ(0)}{y}\right)$ is a polynomial in $z$ with constant term removed.  It follows that $\frac{1}{z}\left(R\left(\frac{\hQ(z)}{y}\right) - R\left(\frac{\hQ(0)}{y}\right)\right)$ is a polynomial in $z$, and thus lies in $\L$.  Thus $\frac{R(u)}{z} \in \L$, and since $\hR/R \in \L$, we have $v = \hR/z \in \L$.

Now, $\hQ(0)$ does not involve $y$, so the quantity $A$ in \eqref{E:PMA} does not involve $y$.  Since we have assumed that $R(u)$ depends on $u$, by Lemma \ref{L:variablesinvolved}, $P(y)$ depends on $y$ as well, so we see that $\frac{R\left(\frac{\hQ(0)}{y}\right)}{z}$ is not divisible by $z$ in $\L$.  Since $z$ is irreducible in $\L$, it follows that $\frac{R\left(\frac{\hQ(0)}{y}\right)}{z} = Cx$, where $C \in \L$ does not depend on $x$, and $\gcd(C,z) = 1$.

However, $f(z)$ is a polynomial in $z$ whose coefficients do not depend on $x$.  Thus we have
$$
\frac{R(u)}{z} \equiv B + Cx \mod z
$$
for $B,C \in \L$ satisfying $\gcd(C,z) = 1$, and $B,C$ do not depend on $x$.  It follows that $\gcd(z,v) = \gcd(z,B+Cx)= 1$.

\medskip
Case 2b: Suppose $\hQ(z)$ does not depend on $z$, so $\hQ(z) = \hQ(0)$.  Then
$$
\hR(u) = \left(\frac{\hQ(0)}{y}\right)^{-p}  \cdot R\left(\frac{\hQ(0)}{y}\right)  \cdot M'
$$
where $M'$ involves only the other cluster variables, and by \eqref{hatF} and Lemma \ref{L:subs}, $p$ is chosen so that $\hR \in \L$ and is not divisible by $\hQ(0)$.  But by the definition of how to obtain $P(y)$ from $R(u)$ by mutation, we have 
$$
R\left(\frac{\hQ(0)}{y}\right) = P(y) \cdot M'' \cdot A
$$
where $A$ is a product of some factors of $\hQ(0)$, and $M''$ is a unit in $\L$.  Since $\hQ(0) = \hQ(z)$ is irreducible by (LP1), we see that $\hR(u)/P(y)$ is a unit in $\L$.  It follows that $v =\hR(u)/z = \hR(u)x/\hP(y) \in \L$ is a unit in $\L$, and $\gcd(z,v) = 1$.
%
%
%
%
%
%
%
%
%
%
\end{proof}

\begin{proof}[Proof of Theorem \ref{T:Laurent}]
 Denote by $t_0 = t$ our original cluster the Laurent polynomial ring $\L(t)$ we are considering.  Let $t_{\head}$ be the cluster containing the cluster variable $w$ we are trying to prove lies in $\L(t)$. 
 Find the mutation path from $t_0$ to $t_{\head}$ which we shall refer to as the {\it {spine}}. Here we consider the exchange graph to be 
a (infinite) regular tree of degree $n$, ignoring possible monodromies. Thus every cluster 
$t_{\head}$ can be assumed to have a unique path from $t_0$. The argument is by induction on the length the spine.  If it has length one, the statement is obvious, and if it has length two it is addressed in Lemma \ref{lem:u}.

Assume now the length of the spine is at least three. Assume that the first two steps from $t_0$ to $t_{\head}$ are 
$$
\overset{x,y}{\bullet} \overset{\hP}{\text{---------}}\overset{z,y}{\bullet} \overset{\hQ}{\text{---------}} \overset{z,u}{\bullet}.
$$
Consider a third mutation, which mutates the same variable as the first step did, obtaining the familiar diagram 
$$
\overset{x,y}{\bullet} \overset{\hP}{\text{---------}}\overset{z,y}{\bullet} \overset{\hQ}{\text{---------}} \overset{z,u}{\bullet} \overset{\hR}{\text{---------}} \overset{v,u}{\bullet}
$$
where the clusters from left to right are $t_0$, $t_1$, $t_2$ and $t_3$. Note that $t_3$ {\it might not} lie on the spine, but it is closer to $t_{\head}$ than $t_0$, and so is $t_1$. By the induction assumption we have $w \in t_{\head}$ lies in $\L(t_1)$ and in $\L(t_3)$. Thus we have two expressions $w = f/z^a$ and $w = g/u^b v^c$, where $f$ and $g$ lie in $\L = \L(t_0)$. By Lemmas \ref{lem:u} and \ref{lem:v} we know $z$ is relatively prime with both $u$ and $v$, which implies $w \in \L$.
\end{proof}

\section{Rank two} \label{sec:r2}
In this section we classify rank two Laurent phenomenon algebras, and give an explicit description of normalized LP algebras of rank two with finitely many seeds.  Let $\A$ be a LP algebra of rank 2, with seeds $\ldots, t_{-1},t_0,t_1,\ldots$ and cluster variables $\ldots,x_{-1},x_0,x_1,x_2,\ldots$ so that $t_i$ contains the cluster variables $\{x_i,x_{i+1}\}$ as in the following:
\begin{equation}\label{E:rank2}
{\text{---------}}\overset{x_0,x_1}{\bullet} {\text{---------}}\overset{x_1,x_2}{\bullet} {\text{---------}} \overset{x_2,x_3}{\bullet} {\text{---------}} \overset{x_3,x_4}{\bullet}
{\text{---------}}
\end{equation}

Note that the seeds and variables may be repeated.  

\begin{example}\label{ex:FZ}
Let $S$ be a coefficient ring and $\FF = \Frac(S[x_1,x_2])$.  Let $q_1,q_2,r_1,r_2 \in S$ and $b,c \in \Z_{\geq 1}$ be such that $r_1 + q_1x^c$ and $r_2+q_2x^c$ are irreducible in $S[x]$.  Let $\A_{b,c}$ be the rank two LP algebra with initial seed $t_1 = \{\{x_1,x_2\},\{r_2+q_2x_2^c,r_1+q_1x_1^b\}\}$.  While we cannot apply Proposition \ref{P:cluster} (unless $r_1,r_2,q_1,q_2$ are variables), nevertheless the cluster algebra $\A_{CA}(t_1)$ with initial seed $t_1$ can naturally be identified with $\A_{b,c}$.  The seeds of $\A_{b,c}$ are of the form
$t_i =\{\{x_i,x_{i+1}\},\{r_{i+1}+q_{i+1}x_{i+1}^{b_{i+1}},r_i+q_ix_{i}^{b_{i}}\}\}$
where $b_i = c$ if $i$ is even, and $b_i = b$ if $i$ is odd, and $r_i,q_i \in S$ satisfy recursions given in \cite[Example 2.5]{CA1}.  
\end{example}

By Theorem \ref{T:Laurent}, we may write
$$
x_m = \frac{S(x_1,x_2)}{x_1^{d_1(m)}x_2^{d_2(m)}}
$$
for a polynomial $S(x_1,x_2) \in S[x_1,x_2]$ not divisible by either $x_1$ or $x_2$.  Following terminology of cluster algebras \cite{CA1} we call the vector 
$$
\delta(m) = d_1(m) \alpha_1 + d_2(m) \alpha_2
$$
the {\it denominator vector} of $x_m$.  Here $\alpha_1$ and $\alpha_2$ are a basis of a two-dimensional lattice $Q \simeq \Z^2$.  Given a rank two Cartan matrix
$$
\left( \begin{array}{cc} 2 & -b \\ -c &2 \end{array} \right)
$$
we have a set $\Phi^+$ of positive real roots, and a set $\Phi^+ \cup \{-\alpha_1,-\alpha_2\}$ of almost positive real roots.  We refer the reader to \cite[Section 6]{CA1} for full details.

\begin{prop}\label{P:den}
Suppose the exchange polynomials of $t_1$ are $F_1 = P(x_2)$ and $F_2 = Q(x_1)$ with degrees $c \geq 1$ and $b \geq 1$ respectively.  Then the set of denominator vectors of $\A$ is exactly $\Phi^+ \cup \{-\alpha_1,-\alpha_2\}$.
\end{prop}
\begin{proof}
We first observe that the condition that $F_1$ and $F_2 $ depend on $x_2$ and $x_1$ implies that all the exchange polynomials of $\A$ depend on the other variable of that cluster, and in particular that $\hF = F$ for all the exchange polynomials of $\A$.  

When $x_2$ is mutated, we have 
$$F_1'(x_0) = \frac{x_0^c F_1(F_2(0)/x_0)}{T}
$$
for $T \in S$ not depending on $x_1$ or $x_2$.  In particular $F_1'$ also has degree $c$ in $x_0$.  It follows easily from this that the denominator vectors $\delta(m)$ depend only on $b$ and $c$, and not on $P(x_2)$ and $Q(x_1)$.  So to compute $\delta(m)$ we may assume we are in the situation of Example \ref{ex:FZ} where the result is established in \cite{CA1}.
\end{proof}
Just as in \cite[Theorem 6.1]{CA1}, each $\delta(m)$ can be computed explicitly, but we shall not need this in the following. Proposition \ref{P:den} does not consider the case where $P(x_2)$ does not depend on $x_2$, or $Q(x_1)$ does not depend on $x_1$.  Instead we have

\begin{prop}\label{P:Q0}
Suppose the exchange polynomials of $t_1$ are $F_1 = P(x_2)$ and $F_2 = Q(x_1)$ and suppose that $Q(x_1)$ does not depend on $x_1$, and furthermore that $P \not \propto Q$.  Then $x_0 \propto x_4$.
%
%
\end{prop}

\begin{proof}
Let $d$ be such that $\hP = P/x_1^d$, and let $k$ be the degree of $P$ as a polynomial in $y$. Clearly $k \geq d$. Then 
$$R \propto \frac{P(\frac{Q}{x_3}) x_3^{k}}{Q^d}$$
where $R=R(x_3)$ is the exchange polynomial for $x_2$ in $t_2 = \{x_2,x_3\}$.
In terms of the $t_1 = \{x_1,x_2\}$ cluster  we have 
$$R \propto \frac{P(x_1)(\frac{Q}{x_1})^{k-d}}{x_1^d},$$ and clearly $k-d$ is the largest power of $\frac{Q}{x_1}$ you can divide it by such that the result is a Laurent polynomial. Thus by Lemma \ref{lem:hcrit} $\hR=R/x_3^{k-d}$.
But then
$$\hR \propto \frac{P(x_1)(\frac{Q}{x_1})^{k-d}}{x_1^d (\frac{Q}{x_1})^{k-d}} = \frac{P(x_1)}{x_1^d} = \hP,$$ and therefore $x_0$ and $x_4$ are differ by a unit in $S$.
\end{proof}

\begin{theorem}\label{thm:rank2}
Suppose the exchange polynomials of $t_1$ are $F_1 = P(x_2)$ and $F_2 = Q(x_1)$ with degrees $c \geq 0$ and $b \geq 0$ respectively, and assume that $c \geq b$.  Then $\A$ is of finite type if and only if either $b = 0$, or $(b,c)$ is equal to one of $(1,1)$, $(1,2)$ and $(1,3)$.
\end{theorem}
\begin{proof}
For simplicity let us denote the initial seed by $t_1 = \{(x,P(y)),(y,Q(x))\}$, so $x_1 = x$ and $x_2 = y$ have exchange polynomials $P(y)$ and $Q(x)$ respectively.

Suppose $b = 0$.  If $P \propto Q$ then $\A$ has a normalization $\A'$ consisting of the following three seeds (see Lemma \ref{L:allequal})
\begin{equation}
\label{E:triangle}
\{(x,P),(y,P)\}, \;\; \{(x,P),(z,P)\}, \;\; \{(z,P),(y,P)\} \;\; \text{where} \;\; z=P/xy.
\end{equation}

If $P \not \propto Q$ then by Proposition \ref{P:Q0} (and with $k,d$ as in Proposition \ref{P:Q0}) our LP algebra has normalization $\A'$ whose four seeds are 
\begin{equation}\label{E:square}
\{(x,P),(y,Q)\}, \;\; \{(z,P),(y,Q)\}, \;\; \{(z,R),(u,Q)\} \;\; \{(x,R),(u,Q)\}
\end{equation}
where $$\hP = P/y^d, \;\; \hQ = Q, \;\; \hR = R/u^{k-d} = \hP.$$

\begin{figure}
\makebox[\textwidth]{
\begin{tabular}{|c|c|}
\hline
\multirow{4}{*}{$b=0$  triangle} & $
\{(x,P),(y,P)\}$\\
&$ \{(x,P),(z,P)\}$\\ 
&$\{(z,P),(y,P)\}$
\\
\cline{2-2}
&$z = P/xy$
\\
\hline
\hline
\multirow{5}{*}{$b=0$  square} & $
\{(x,P),(y,Q)\}$ \\ &$\{(z,P),(y,Q)\}$\\ &$\{(z,R),(u,Q)\}$ \\ &$\{(x,R),(u,Q)\}$ \\
\cline{2-2}
&$\hP = P/y^d, \;\; \hQ = Q, \;\; \hR = R/u^{k-d} = \hP.$
\\
\hline
\hline
\multirow{7}{*}{$(b,c)=(1,1)$ pentagon}
&$\{(x, Ay+BE),(y, Cx+DE)\}$\\
&$\{(z, Ay+EB),(y, Dz+CB)\}$\\
&$\{(z, Eu+AC),(u, Dz+BC)\}$\\
&$\{(t, Eu+CA),(u, Bt+DA)\}$\\
&$\{(t, Cx+ED),(x, Bt+AD)\}$\\
\cline{2-2}
&$\gcd(A,B) = \gcd(A,E) =\gcd(B,D)= \gcd(C,D) = \gcd(C,E) = 1$\\
\cline{2-2}
&mutation generated by $(BC)(DE)$ and $(AB)(DE)$\\
\hline
\hline
\multirow{9}{*}{$(b,c)=(2,1)$ hexagon}
&$\{(x, Ay^2+BGy+CFG^2),(y, Dx+EFG)\}$\\
&$\{(z, Ay^2+BGy+FCG^2),(y, Ez+DCG)\}$\\
&$\{(z, Fu^2+BDu+ACD^2),(u, Ez+GCD)\}
$ \\
&$\{(v, Fu^2+BDu+CAD^2),(u, Gv+EAD)\}
$\\
&$\{(v, Ct^2+BEt+FAE^2),(t, Gv+DAE)\}$\\
&$\{(x, Ct^2+BEt+AFE^2),(t, Dx+GFE)\}$\\
\cline{2-2}
&$\gcd(A,G)=\gcd(C,E)=\gcd(D,E)=\gcd(D,F)=\gcd(D,G)=\gcd(E,G)$\\
& $=\gcd(A,B,C)=\gcd(A,B,F)=\gcd(B,C,F) =1$\\
\cline{2-2}
& mutation generated by $(CF)(DE)$ and $(AC)(EG)$
\\
\hline
\hline
\multirow{14}{*}{$(b,c)=(3,1)$ octagon}
&$\{(x, Ay^3+BKLy^2+CHK^2L^2y+DGH^2K^3L^3),(y, Ex+FGHKL^2)\}$\\
&$\{(z, Ay^3+BLKy^2+CHL^2K^2y+GDH^2L^3K^3),(y, Fz+EDHLK^2)\}$\\
&$\{(z, Gu^3+CEKu^2+BDE^2K^2u+AHD^2E^3K^3),(u, Fz+LHDEK^2)\}$\\
&$\{(w,Gu^3+CKEu^2+BDK^2E^2u+HAD^2K^3E^3),(u, Lw+FADKE^2)\}$\\
&$\{(w, Hs^3+BFEs^2+CAF^2E^2s+GDA^2F^3E^3),(s, Lw+KDAFE^2)\}$\\
&$\{(v, Hs^3+BEFs^2+CAE^2F^2s+DGA^2E^3F^3),(s, Kv+LGAEF^2)\}$\\
&$\{(v, Dt^3+CLFt^2+BGL^2F^2t+HAG^2L^3F^3),(t, Kv+EAGLF^2)\}$\\
&$\{(x, Dt^3+CFLt^2+BGF^2L^2t+AHG^2F^3L^3),(t, Ex+KHGFL^2)\}$\\
\cline{2-2}
&
$\gcd(A,K)=\gcd(A,L)=\gcd(D,F)=\gcd(D,L)=\gcd(E,F)=\gcd(E,G)=\gcd(E,H)$\\
&$=\gcd(E,K)=\gcd(E,L)=\gcd(F,H)=\gcd(F,K)=\gcd(F,L)=\gcd(G,K)=\gcd(K,L)$\\
&$=\gcd(A,B,H)=\gcd(C,D,G)$\\
&$=\gcd(A,B,C,D)=\gcd(A,B,C,G)=\gcd(B,C,D,H)=\gcd(B,C,G,H)=1$\\
\cline{2-2}
&  mutation generated by $(DG)(EF)(KL)$ and $(AD)(BC)(FK)(GH)$\\
\hline
\end{tabular}
}
\caption{Finite type normalized LP algebras of rank two}
\label{fig:finite}
\end{figure}

\medskip

Now suppose that $b > 0$ and thus also $c > 0$.  We apply Proposition \ref{P:den}.  It is clear that the denominator vector is an invariant of the equivalence class of a seed.  Thus $\A$ can be of finite type only if the set $\Phi^+ \cup \{-\alpha_1,-\alpha_2\}$ of almost positive roots is finite, which happens if and only if the Cartan matrix $$
\left( \begin{array}{cc} 2 & -b \\ -c &2 \end{array} \right)
$$
is of finite type.  Thus we are reduced to $(b,c)$ being equal to one of $(1,1)$, $(1,2)$ and $(1,3)$.  For each of these cases, we shall now construct a normalized LP algebra $\A'$ with initial seed $t_1 = \{(x,P(y)),(y,Q(x))\}$.  
By Lemma \ref{lem:normal}, this normalized LP algebra $\A'$ will be the normalization of $\A$.

Before we begin, we note that for all the exchange polynomials we shall encounter, we have $\hF = F$.

\medskip

\noindent {\bf{The case $(b,c) = (1,1)$}}.

In this case  we have
$$t_1 = \{(x, A_0y+B_0),(y, C_0x+D_0)\}$$
where $A_0,B_0,C_0,D_0 \in S$, and $A_0,C_0 \neq 0$.  Let $E = \gcd(B_0,D_0)$ (which is only defined up to a unit in $S$).  Renaming the coefficients, we shall write the initial seed now as 
$$\{(x, Ay+BE),(y, Cx+DE)\},$$ 
where $A:= A_0$, $B:=B_0/E$, $C:=C_0$, and $D:=D_0/E$, and $B,D$ now satisfy $\gcd(B,D)=1$.  Also the irreducibility of $Ay+BE$ and $Cx+DE$ is equivalent to $\gcd(A,B) = \gcd(A,E) = \gcd(C,E) = \gcd(C,D) = 1$, and there are no further restrictions on $A,B,C,D,E$.

Mutation at $x$ gives the seed 
$$\{(z, Ay+BE),(y, Dz+BC)\}.$$
Indeed, the mutation rule tells us to substitute $z = BE/x$ into $Cx+DE$, and then kill any common factors with $BE$. Thus we kill the factor $E$ in $BCE+DEz$, but no other factor 
since $D$ is relatively prime with $B$ and $C$, by construction.

This seed has exactly the same form as the initial seed, except the coefficients $A,B,C,D,E$ are permuted as follows:
$$B \longleftrightarrow E, \;\;\; C \longleftrightarrow D.$$ 
Furthermore this relabeling just permutes the five relatively prime pairs $$\{(B,D), (A,B),(A,E),(C,E),(C,D)\}.$$  As a result, we know that the next mutation 
will be identical to the first one, just with a different permutation of coefficients. 
Proceeding in this fashion, one checks that the list of clusters has the form given in Figure \ref{fig:finite}.  An example is the above calculation obtaining cluster 
$\{(z, Ay+BE),(y, Dz+BC)\}$ from the cluster $\{(x, Ay+BE),(y, Cx+DE)\}$.
The fact the cluster variables ``wrap around'' after five mutations is a simple computer calculation with rational functions.  For example, it says that 
$$
z = \dfrac{Ay+BE}{x} \qquad u = \dfrac{Dz+CB}{y} \qquad t = \dfrac{Eu+AC}{z}
$$
gives $t =\dfrac{Cx+ED}{y}$.  It follows from denominator vector considerations (Proposition \ref{P:den}) that all of $x,y,z,u,t$ are distinct even up to units. 

Note that the subgroup of the permutation group of $\{A,B,C,D,E\}$ generated by the involutions $(BC)(DE)$ and $(AB)(DE)$ is a dihedral group of order 10.  The element in the center of this subgroup acts on the seeds by swapping the two cluster variables and exchange polynomials.

\medskip

\noindent
{\bf{The case $(b,c) = (1,2)$}}.

The initial seed in this case looks like 
$$\{(x, A_0y^2+B_0y+C_0),(y, D_0x+E_0)\}.$$
Let $F_0 = \gcd(C_0,E_0)$, and define $C_1:=C_0/F_0$ and $E:=E_0/F_0$, and let $G=\gcd(C_1,B_0,F_0)$, and set $C:=C_1/G$, $B=B_0/G$, and $F:=F_0/G$.
This writes the seed in the form
$$t_1 = \{(x, Ay^2+BGy+CFG^2),(y, Dx+EFG)\}$$ and by construction we have that $\gcd(CG,E)=\gcd(C,B,F)=1$.  Furthermore, the irreducibility of $Ay^2+BGy+CFG^2$ and $Dx+EFG$ imply that $\gcd(A,BG,CFG) = \gcd(D,EFG) = 1$.  Together these $\gcd$ conditions are equivalent to the relatively prime pairs and triples listed in Figure \ref{fig:finite}.  Note that these $\gcd$ conditions do not imply that $Ay^2+BGy+CFG^2$ does not factor into two linear factors.  This is a condition that is separately imposed.

Mutation of $t_1$ at $x$ produces the seed 
$$\{(z, Ay^2+BGy+CFG^2),(y, Ez+CDG)\}.$$
This follows from the definitions, together with the observation that $Ez+CDG$ has no common factor with $(F_x)|_{y\leftarrow 0} = CFG^2$, since $\gcd(E,CDG) = 1$.  This new seed is identical to the original seed with coefficients permuted as follows:
$$F \longleftrightarrow C, \;\;\; E \longleftrightarrow D.$$
Furthermore, this relabeling just permutes the relatively prime pairs and triples.

Similarly, the mutation at $y$ of $t_1$ produces
$$\{(x, Ct^2+BEt+AFE^2),(t, Dx+EFG)\}.$$
To see this we note that $(F_y)|_{x\leftarrow 0} = EFG$, and substituting $y = EFG/t$  into $F_x = Ay^2+BGy+CFG^2$ we obtain the Laurent polynomial $A(EFG)^2 t^{-2} + BEFG^2 t^{-1} +CFG^2$.  We then divide by $FG^2$ and multiply by $t^2$ to obtain $F'_x = Ct^2+BEt+AFE^2$.  Note that there are no further common factors among the coefficients because we have $\gcd(C,BE,AFE^2) = 1$.  This seed is identical to the original seed with coefficients permuted as follows:
$$A \longleftrightarrow C, \;\;\; E \longleftrightarrow G.$$
Again, this relabeling just permutes the relatively prime pairs and triples.

Since the form of the seed always remains the same, it is easy to repeatedly mutate it.
An involved but straightforward computer calculation then checks that after the six mutations one indeed comes back to the original variables.  One checks that the resulting list of clusters is as given in Figure \ref{fig:finite}.  It follows from denominator vector considerations (Proposition \ref{P:den}) that all of $x,y,z,u,v,t$ are distinct even up to units.  

In this case the subgroup of the permutation group on $\{A,B,C,D,E,F,G\}$ generated by the involutions $(CF)(DE)$ and $(AC)(EG)$ has order six.

\medskip

\noindent
{\bf{The case $(b,c) = (1,3)$}}.

The initial seed in this case looks like 
$$\{(x, A_0y^3+B_0y^2+C_0y+D_0),(y, E_0x+F_0)\}$$
where $A_0,B_0,C_0,D_0,E_0,F_0 \in S$.

Let $G_0 = \gcd(D_0,F_0)$ and define $D_1:=D_0/G_0$ and $F=F_0/G_0$.  Then let $H_0=\gcd(D_1,C_0,G_0)$, and define $D_2:=D_1/H_0$, $C_1:=C_0/H_0$ and $G_1=G_0/H_0$.  Then let $K=\gcd(D_2,C_1,B_0,H_0)$ and define $D:=D_2/K$, $C_2:=C_1/K$, $B_1:=B_0/K$, and $H_1:=H_0/K$.  Finally, let $L=\gcd(G_1,C_2,B_1,H_1)$ and define $G:=G_1/L$, $C:=C_2/L$, $B:=B_1/L$, and $H:=H_1/L$. As a result, the seed can be written in the form
$$t_1 = \{(x, Ay^3+BKLy^2+CHK^2L^2y+DGH^2K^3L^3),(y, Ex+FGHKL^2)\}$$ 
and we know that $$\gcd(DHK^2L,F)=\gcd(DK,CKL,GL)=\gcd(D,CL,BL,HL)=\gcd(G,C,B,H)=1.$$
Also the irreducibility of the exchange polynomials gives
$$
\gcd(A,BKL,CHK^2L^2,DGH^2K^3L^3) = \gcd(E,FGHKL^2) = 1.
$$
Together these $\gcd$ conditions are equivalent to the relatively prime pairs, triples, and quadruples listed in Figure \ref{fig:finite}.

We claim that mutation at $x$ produces the seed 
$$t_2 = \{(z, Ay^3+BKLy^2+CHK^2L^2y+DGH^2K^3L^3),(y, Fz+DEHLK^2)\}.$$
To see this, we substitute $x = (F_x)|_{y \leftarrow 0}/z = DGH^2K^3L^3/z$ into $F_y$ to obtain the Laurent polynomial $DEGH^2K^3L^3z^{-1}+FGHKL^2 = (DEHLK^2z^{-1}+F)GHKL^2$.  To check that there are no common factors between $DEHLK^2z^{-1}+F$ and $DGH^2K^3L^3$ it is enough to verify that $\gcd(F,DEHLK^2)=1$, which follows from the $\gcd$ conditions listed in Figure \ref{fig:finite}.  This new seed is identical to the original seed with coefficients permuted as follows:
$$K \longleftrightarrow L, \;\;\; E \longleftrightarrow F, \;\;\; G \longleftrightarrow D.$$
Furthermore, this relabeling just permutes the $\gcd$ conditions.

Similarly, the mutation at $y$ of the original seed produces
$$\{(x, Dt^3+CFLt^2+BGF^2L^2t+AHG^2F^3L^3),(t, Ex+FGHKL^2)\}.$$
To see that $\gcd(D,CFL,BGF^2L^2,AHG^2F^3L^3) = 1$, we use the conditions 
$$
\gcd(D,F) =\gcd(D,L) =\gcd(C,D,G) = \gcd(A,B,C,D) = \gcd(B,C,D,H) = 1$$
from Figure \ref{fig:finite}.
The new seed is identical to the original seed with coefficients permuted as follows:
$$B \longleftrightarrow C, \;\;\; K \longleftrightarrow F, \;\;\; G \longleftrightarrow H, \;\;\; A \longleftrightarrow D.$$
Furthermore, this relabeling just permutes the $\gcd$ conditions.

Again, we know that we can just proceed mutating and we will be obtaining similar looking clusters where coefficients are just permuted as described above.  One checks that the resulting list of clusters is as given in Figure \ref{fig:finite}.  The check that the cluster variables correctly wrap around is now a very involved computation with rational functions, which can be verified by computer.  In this case, the subgroup of the permutation group on $\{A,B,C,D,E,F,G,H,K,L\}$ generated by $(DG)(EF)(KL)$ and $(AD)(BC)(FK)(GH)$ has order 8.
\end{proof}

As a corollary we have
\begin{cor}\label{cor:rank2}
A rank two LP algebra $\A$ is of finite type if and only if one of the following equivalent conditions hold:
\begin{itemize}
\item
$\A$ has finitely many equivalence classes of seeds;
\item
$\A$ has finitely many distinct cluster variables, up to units;
\item
$\A$ has finitely many distinct denominator vectors with respect to some seed.
\end{itemize}
\end{cor}

\begin{proof}
 The ``only if" direction follows from Theorem \ref{thm:rank2}, since all the three properties hold for the LP algebras listed in Figure \ref{fig:finite}, and all three properties hold for a LP algebra $\A$ if and only if it holds for the normalization $\A'$ of $\A$.

For the ``if" direction, we first observe that finitely many equivalence classes of seeds implies finitely many distinct cluster variables, up to units, which in turn implies finitely many distinct denominator vectors with respect to any seed. So suppose $\A$ has finitely many distinct denominator vectors with respect to some seed.  Then in the proof of Theorem \ref{thm:rank2} we have constructed a normalization $\A'$ of $\A$ with finitely many seeds.  Thus $\A$ is finite type.  
\end{proof}

A rank two LP algebra of infinite type has an exchange graph which is a doubly-infinite path.  A normalized rank two LP algebra of finite type has an exchange graph which is a triangle, square, pentagon, hexagon, or octagon as described in Figure \ref{fig:finite}.  Only triangles do not occur as exchange graphs of cluster algebras of rank two.  However, the exchange graphs of finite type LP algebras in higher rank is vastly richer than those of cluster algebras, as we shall partly explore in \cite{LP2}.

\begin{theorem}
Suppose $\A$ is a rank two normalized LP algebra of finite type.  Then the list of seeds of $\A$ has the form given in Figure \ref{fig:finite}.
\end{theorem}
\begin{proof}
Let $t_1$ be a seed of $\A$.  In the proof of Theorem \ref{thm:rank2} we have constructed a normalized LP algebra with initial seed $t_1$, and list of seeds given by Figure \ref{fig:finite}.  The two LP algebras will be normalizations of each other by Lemma \ref{lem:normal}.  By the definition of a normalized LP algebra, it follows that there is a bijection $t \leftrightarrow t' = \phi(t)$ between the seeds of $\A$ and $\A'$, and that under this bijection $t$ and $\phi(t)$ are equivalent.  In other words, the only possible discrepancy between the seeds of $\A$ and the list of seeds given by Figure \ref{fig:finite} is that the seeds have have been replaced by equivalent ones.  We shall show that this discrepancy can always be obtained by modifying the coefficients $A,B,C,\ldots$ used in Figure \ref{fig:finite} by units.  For the $b=0$ cases, the situation is trivial. 

Suppose $(b,c) = (1,1)$ and let the seed $t_1$ of $\A$ be $\{(x, Ay+BE),(y, Cx+DE)\}$.  This uniquely determines the cluster variables $z$ and $t$ in adjacent seeds, though the exchange polynomials in adjacent seeds are only determined up to units.  Nevertheless, once we know the last cluster variable $u$ in $\A$, all the exchange polynomials are determined.  So in fact, there is one degree of freedom, and this corresponds to the degree of freedom in the factorizations $BE$ and $DE$: we can modify $E$ by a unit and modify $B$ and $D$ by the inverse unit.  
Let $u'$ denote the corresponding cluster variable in $\A'$.  Since $u' = (Dz+CB)/y = (ADy+BCx+BDE)/xy$, we see that modifying $E$ by a unit and modifying $B$ and $D$ by the inverse unit indeed modifies $u'$ by an arbitrary unit.

Now suppose $(b,c)=(1,2)$.  Let $\A$ have initial seed $\{(x, Ay^2+BHy+CFH^2),(y, Dx+EFH)\}$.  Let $\A'$ be the normalized LP algebra with list of seeds given in Figure \ref{fig:finite}.  Then the cluster variables $u'$ and $v'$ in $\A'$ may differ from the cluster variables $u$ and $v$ in $\A$ by units.  Indeed, if we modify $F$ by a unit $\alpha$ and $C, E$ by $\alpha^{-1}$ (not changing the initial seed), we find that $u'$ is modified by $\alpha^{-1}$.  Similarly if we modify $H$ by a unit $\beta$ and $B,C,F$ by $\beta^{-1}$ (again not changing the initial seed), then $v'$ is modified by $\beta^{-1}$.  So we conclude that $\A$ is indeed of the form in Figure \ref{fig:finite}.

Finally let $(b,c)=(1,3)$.  We proceed in the same way.  Modifying $F,D$ by $\alpha$ and $G$ by $\alpha^{-1}$ changes $u', w', v', s'$ by $\alpha,\alpha^2,\alpha,\alpha$ respectively.  Modifying $D,C,G$ by $\beta$ and $H$ by $\beta^{-1}$ changes $u',w',v',s'$ by $1,\beta,\beta,\beta$ respectively.  Modifying $D,C,B,H$ by $\gamma$ and $K$ by $\gamma^{-1}$ changes only $v'$ by $\gamma$.  Modifying $G,C,B,H$ by $\delta$ and $L$ by $\delta^{-1}$ changes $w'$ by $\delta$.  This allows us to modify $u',w',v',s'$ by arbitrary units, completing the proof.
\end{proof}

\begin{remark}
 Our classification of the rank two finite type LP algebras is essentially what Fomin and Zelevinsky (in the context of cluster algebras) call universal coefficients, see \cite[Section 12]{CA4}.
\end{remark}

\section{Examples} \label{sec:ex}
\subsection{The Gale Robinson LP algebra}\label{sec:GaleRobinson}
In \cite{FZLP}, Fomin and Zelevinsky studied a number of multi-dimensional recurrence sequences, establishing the Laurent phenomenon.  These include: the cube recurrence, the Somos sequences, and the Gale-Robinson sequence.  As an example we show how the following case of the Gale-Robinson recurrence fits into our framework:
\begin{equation}\label{E:Gale}
y_iy_{i+6} = y_{i+3}^2 + y_{i+2}y_{i+4}+y_{i+1}y_{i+5}.
\end{equation}
The recurrence \eqref{E:Gale} defines all $y_i$ given the initial $y_1,y_2,\ldots,y_6$.  

We take $S =\Z$ and $\FF = \Q(y_1,y_2,\ldots,y_6)$.  As initial seed we have
\begin{gather*}
t_1 = \{(y_1,y_4^2+y_3 y_5+y_2 y_6),(y_2,y_3 y_4^2+y_3^2 y_5+y_1y_5^2+y_1y_4 y_6),\\(y_3,y_2 y_4^2 y_5+y_1 y_4 y_5^2+y_1 y_4^2 y_6+y_2^2 y_5 y_6+y_1 y_2 y_6^2),  
(y_4,y_2 y_3^2 y_5+y_1 y_2 y_5^2+y_2^2 y_3 y_6+y_1 y_3^2 y_6+y_1^2 y_5 y_6),\\(y_5,y_3^2 y_4+y_1 y_3 y_6+ y_2y_4^2+y^2_2 y_6),(y_6,y_3^2+y_2 y_4+y_1 y_5)\}
\end{gather*}
where for clarity we have listed the cluster variables next to the corresponding exchange polynomial.  It is not difficult to check that all the exchange polynomials are irreducible, and satisfy $\hF = F$.

Mutating at $y_1$ we obtain the seed
\begin{gather*}
t_2 = \{(y_7,y_4^2+y_3 y_5+y_2 y_6),(y_2,y_5^2+y_4 y_6+y_3 y_7),(y_3,y_4 y_5^2+y_4^2 y_6+y_2 y_6^2+y_2y_5 y_7),
\\
(y_4,y_3 y^2_5 y_6+y_2 y_5 y_6^2+y_2  y^2_5 y_7+,y_3^2 y_6 y_7+y_2y_3 y_7^2),
\\
(y_5,y_3 y_4^2y_6+y_2 y_3y^2_6+y_3^2 y_4 y_7+y_2y_4^2y_7+y^2_2 y_6y_7),
(y_6,y_4^2 y_5+y_2 y_4 y_7+y_3 y_5^2+y^2_3 y_7)\}
\end{gather*}
where $y_7$ is the new cluster variable, related to $y_1$ via the formula
$$
y_1y_7 = y_4^2+y_3 y_5+y_2 y_6.
$$
Note that $t_2$ can be obtained from $t_1$ by reindexing the $y_i$'s and thus $t_1$ and $t_2$ are similar seeds in the language of Section \ref{sec:finitetype}.  It follows that if we mutate $t_2$ at $y_2$, and so on, the form of the seeds will remain the same, and we will generate the recurrence \eqref{E:Gale}.  By Theorem \ref{T:Laurent}, it then follows that all the $y_i$ defined by \eqref{E:Gale} are Laurent polynomials in $y_1,y_2,\ldots,y_6$.  It is not however clear how to describe all the seeds of this LP algebra.

\begin{remark}
Essentially all the examples in \cite{FZLP} can be fit into our framework in this way: the exchange polynomials of the initial seed can be calculated by repeatedly 
mutating the polynomial defining the recurrence relation ($y_4^2+y_3 y_5+y_2 y_6$ in our example above).  One technical point is that we require the recurrence 
polynomial to be irreducible.  This can usually be overcome by introducing coefficients: for example $1+x^3$ is reducible in $\Z[x]$, but $A+x^3$ is irreducible in 
$\Z[x,A]$.
\end{remark}

\begin{remark}
The work of Andrew Hone on the Laurent phenomenon beyond the cluster case \cite{Ho} contains more examples that fit into our Laurent phenomenon algebras setting.
His recurrence (5.1) is one such example.
\end{remark}

\subsection{LP algebras of finite mutation type} \label{sec:finitemutation}
As the following example shows, there are LP algebras of finite mutation type which do not fall into the cluster setting.

Take the coefficient ring $S = \Z$, ambient field $\FF = \Q(y_1,y_2,y_3)$ and initial seed $$t = \{(y_1, y_2+y_3+1),(y_2, y_1^2+y_1y_3+y_3^2), (y_3, y_2+y_1+1)\}.$$ Then one obtains the exchange graph shown in Figure \ref{fig:lp1}.  The initial seed corresponds to the vertex shared by the three bricks labeled $y_1$, $y_2$, and $y_3$.

\begin{figure}[ht]
    \begin{center}
\setlength{\unitlength}{2pt}
\begin{picture}(240,60)(-10,0)
\thicklines
\put(5,10){\line(1,0){220}}
\put(5,30){\line(1,0){220}}
\put(5,50){\line(1,0){220}}
\put(25,10){\line(0,1){20}}
\put(85,10){\line(0,1){20}}
\put(145,10){\line(0,1){20}}
\put(205,10){\line(0,1){20}}
\put(10,30){\line(0,1){20}}
\put(40,30){\line(0,1){20}}
\put(70,30){\line(0,1){20}}
\put(100,30){\line(0,1){20}}
\put(130,30){\line(0,1){20}}
\put(160,30){\line(0,1){20}}
\put(190,30){\line(0,1){20}}
\put(220,30){\line(0,1){20}}
\put(113,0){$z$}
\put(113,58){$w$}
\put(23,38){$y_0$}
\put(53,18){$y_1$}
\put(83,38){$y_2$}
\put(113,18){$y_3$}
\put(143,38){$y_4$}
\put(173,18){$y_5$}
\put(203,38){$y_6$}
\put(53,38){$x_1$}
\put(113,38){$x_3$}
\put(173,38){$x_5$}
\end{picture}
\qquad
    \end{center} 
    \caption{Two-layer brick wall with two brick sizes.
}
    \label{fig:lp1}
\end{figure}
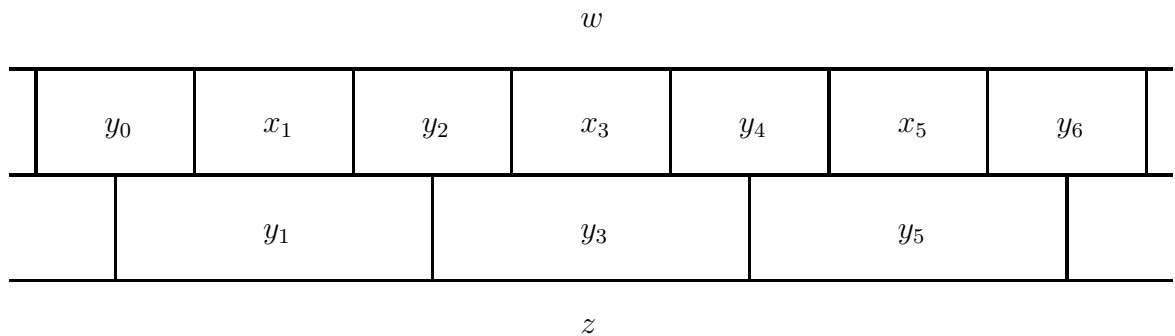
Any seed in this LP algebra is similar to either the initial seed, or one of the following three seeds:
$$\{(y_1, 1+x_1),(y_2, 1+y_1(2+x_1)+y_1^2(1+x_1+x_1^2)), (x_1,1+y_1+y_2)\},$$
$$\{(y_1, y_3^2+z+y_3 z),(z, y_1^2+y_1y_3+y_3^2), (y_3, y_1^2+z+y_1z)\},$$
$$\{(x_1,1+w+y_2w),(y_2, 3x_1^3+x_1^4+3x_1(1+w)+(1+w)^2+x_1^2(4+w)), (w, 1+x_1)\}.$$
The similarity type of a seed depends on the shapes of the bricks that the corresponding vertex lies in.  We encourage the reader to compare this example with the two-layer brick wall example in \cite{CA1}.

\subsection{Chekhov-Shapiro LP algebras}
In their work on Teichm\"{u}ller spaces 
of Riemann surfaces with orbifold points Chekhov and Shapiro \cite{ChSh} study a generalization of cluster algebras, which they call {\it generalized cluster algebras}.
They show that their algebras satisfy the Laurent phenomenon and have the same finite type classification as cluster 
algebras. 

LP algebras generalize Chekhov and Shapiro's algebras in a similar manner to the way LP algebras generalize cluster algebras (as in Section \ref{sec:LPCL}); 
that is, the dynamics studied in \cite{ChSh} are a special case of LP algebra dynamics with some assumption on the non-degeneracy of coefficients.  
All the exchange polynomials of a cluster algebra are binomials, so in particular the Newton polytope of the exchange polynomials are line segments.  
The Chekhov-Shapiro LP algebras are essentially those LP algebras $\A$ for which there is a cluster algebra $\A'$, together with a bijection between the seeds 
of $\A$ and $\A'$ under which the Newton polytopes of all the exchange polynomials are identical. 

\begin{example} [\cite{ChSh}]
Let $S = \Z[A,B,C,P,Q]$, $\FF=\Frac(S)(x,y)$.  Consider the initial seed $t = \{(x, A+By+Cy^2),(y, Q+Px)\}$.  The LP algebra $\A(t)$ with initial seed $t$ is a Chekhov-Shapiro LP algebra, and it was shown in \cite{ChSh} that $\A(t)$ has the same cluster complex as the type $B_2$ cluster algebra.  The Newton polytope of $F_x$ is the line segment connecting the lattice points $(0,0)$ with $(0,2)$, even though $F_x$ is not a binomial.  This agrees with the Newton polytopes of the initial seed of the cluster algebra of type $B_2$, which can be taken to be $t' = \{(x, A+Cy^2),(y, Q+Px)\}$.

Note that in this case the close connection between $\A(t)$ and $\A(t')$ also follows from our Theorem \ref{thm:rank2}.
\end{example}

The Chekhov-Shapiro LP algebras are a much narrower generalization of cluster algebras than the LP algebras in general. On the other hand, Chekhov-Shapiro LP algebras resemble cluster algebras more closely, and thus potentially more properties of cluster algebras extend to them. 

\subsection{Linear LP algebras}

Let $\Gamma$ be a directed, multiplicity-free, loopless graph on the vertex set $[n] = \{1,2,\ldots,n\}$. Thus, every edge $i \longrightarrow j$ is 
either present with multiplicity one or absent, for each ordered pair $(i,j)$, $i \not = j$.  Define the initial seed $t_\Gamma$ with variables $ (X_1,\ldots,X_n)$ and exchange polynomials 
$F_i = A_i + \sum_{i \to j} X_j$, where $i \to j$ denotes an edge in $\Gamma$. The following theorem is proved in \cite{LP2}.

\begin{theorem} \cite{LP2}
 For any directed graph $\Gamma$, the LP algebra $\A_{\Gamma}$ with initial seed $t_{\Gamma}$ is of finite type.
\end{theorem}
As we already saw when we looked at rank $2$, there are LP algebras of finite type which do not possess a linear seed. 

Let us identify subsets of vertices of $\Gamma$ with the corresponding 
induced subgraphs, for example we shall talk about {\it {strongly connected}} subsets, and so on.  Let $\I \subset 2^{[n]}$ denote the collection of strongly-connected subsets of $\Gamma$. A family of subsets $\s = \{I_1, \ldots, I_k\} \in \I$ is {\it nested} if
\begin{itemize}
 \item for any pair $I_i, I_j$ either one of them lies inside the other, or they are disjoint;
 \item for any tuple of disjoint $I_j$-s, they are the strongly connected components of their union.
\end{itemize}
The {\it support} $S$ of a nested family $\s = \{I_1, \ldots, I_k\}$ is $S = \bigcup I_j$.  A nested family is {\it maximal} if it is not properly contained in another nested family with the same support.

\begin{theorem} \cite{LP2}
Non-initial cluster variables in $\A_{\Gamma}$ are in bijection with elements of $\I$. The clusters of $\A_{\Gamma}$ are in bijection with maximal nested families of $\Gamma$.
\end{theorem}

\begin{example}
 Consider the graph $\Gamma$ on four vertices with edges $1 \longrightarrow 2$, $2 \longrightarrow 1$, $1 \longrightarrow 3$, $3 \longrightarrow 1$, $3 \longrightarrow 2$, $2 \longrightarrow 3$, $1 \longrightarrow 4$, $3 \longrightarrow 4$, $4 \longrightarrow 2$, shown in Figure \ref{fig:lp2}.
Then the initial seed is given by 
$$t_\Gamma = \{(X_1, A_1 + X_2 + X_3 + X_4),(X_2, A_2 + X_1 +X_3),(X_3, A_3 + X_1 + X_2 + X_4),(X_4, A_4 + X_2)\}.$$
The resulting LP algebra has $15$ cluster variables and $46$ clusters.
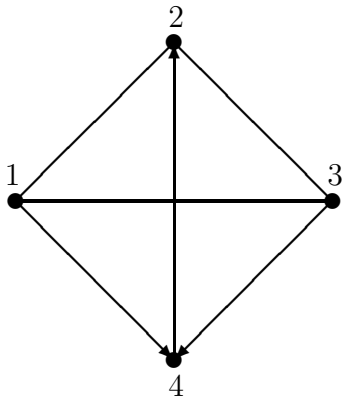
\begin{figure}[ht]
    \begin{center}
\setlength{\unitlength}{2pt}
\begin{picture}(60,60)(0,0)
\thicklines
\put(0,30){\line(1,1){30}}
\put(0,30){\vector(1,-1){30}}
\put(0,30){\line(1,0){60}}
\put(30,60){\line(1,-1){30}}
\put(60,30){\vector(-1,-1){30}}
\put(30,0){\vector(0,1){60}}
\put(0,30){\circle*{3}}
\put(-2,33){$1$}
\put(30,60){\circle*{3}}
\put(29,63){$2$}
\put(60,30){\circle*{3}}
\put(59,33){$3$}
\put(30,0){\circle*{3}}
\put(29,-7){$4$}
\end{picture}
\qquad
    \end{center} 
    \caption{The example graph
}
    \label{fig:lp2}
\end{figure}

\end{example}

It turns out the cluster complex of $\A_\Gamma$ contains inside it the nested complex studied in \cite{Pos,FS, Zel}.  In particular, there is a LP algebra $\A'_\Gamma$, obtained from $\A_\Gamma$ by freezing, such that the exchange graph of $\A'_\Gamma$ is the $1$-dimensional skeleton of a polytope known as a nestohedron \cite{Pos, Zel}.

We refer the reader to \cite{LP2} for full details on the structure of linear LP algebras arising from graphs.

\subsection{LP algebras arising from electrical Lie groups}

Consider a wiring diagram $\W$ in a disk: a collection of simple curves called {\it wires} embedded into a disk, with endpoints on the boundary of the disk, such that no two curves intersect more than once, and all intersection points are transversal.  The wires subdivide the disk into regions, and a region is {\it internal} if it is bounded completely by wires.  We assign a cluster variable to each internal region and a frozen variable to each non-internal region. 

For each internal region with variable $a$ define an exchange polynomial $F_a$ as follows. Each region $R$ adjacent to the $a$ region is either a {\it {corner region}} if it shares only a vertex with region $a$, or a {\it {side region}} if it shares an edge with region $a$.  Associate to each corner region $R$ the monomial $m_R$ obtained by multiplying 
its variable with the variables of all side regions that are not adjacent to it. Let $F_a = \sum_R m_R$ be the sum of these monomials over all the corner regions. For example, the monomials one needs to sum for a pentagonal 
region are schematically shown in Figure \ref{fig:lp3}.

\begin{figure}[ht]
    \begin{center}
\setlength{\unitlength}{2pt}
\begin{picture}(240,60)(10,0)
\thicklines
\put(5,25){\line(1,1){24}}
\put(5,35){\line(1,-2){13}}
\put(10,15){\line(1,0){30}}
\put(45,35){\line(-1,-2){13}}
\put(18,49){\line(1,-1){27}}
\put(23,50){\circle*{3}}
\put(41,19){\circle*{3}}
\put(25,10){\circle*{3}}
\put(8,20){\circle*{3}}

\put(55,25){\line(1,1){24}}
\put(55,35){\line(1,-2){13}}
\put(60,15){\line(1,0){30}}
\put(95,35){\line(-1,-2){13}}
\put(68,49){\line(1,-1){27}}
\put(95,27){\circle*{3}}
\put(75,10){\circle*{3}}
\put(58,20){\circle*{3}}
\put(63,39){\circle*{3}}

\put(105,25){\line(1,1){24}}
\put(105,35){\line(1,-2){13}}
\put(110,15){\line(1,0){30}}
\put(145,35){\line(-1,-2){13}}
\put(118,49){\line(1,-1){27}}
\put(113,39){\circle*{3}}
\put(138,10){\circle*{3}}
\put(136,38){\circle*{3}}
\put(108,20){\circle*{3}}

\put(155,25){\line(1,1){24}}
\put(155,35){\line(1,-2){13}}
\put(160,15){\line(1,0){30}}
\put(195,35){\line(-1,-2){13}}
\put(168,49){\line(1,-1){27}}
\put(163,39){\circle*{3}}
\put(186,38){\circle*{3}}
\put(163,11){\circle*{3}}
\put(163,39){\circle*{3}}
\put(191,19){\circle*{3}}

\put(205,25){\line(1,1){24}}
\put(205,35){\line(1,-2){13}}
\put(210,15){\line(1,0){30}}
\put(245,35){\line(-1,-2){13}}
\put(218,49){\line(1,-1){27}}
\put(241,19){\circle*{3}}
\put(225,10){\circle*{3}}
\put(203,29){\circle*{3}}
\put(236,38){\circle*{3}}

\end{picture}
\qquad
    \end{center} 
    \caption{The exchange polynomial of a pentagonal region is homogeneous of degree four and has five terms.
}
    \label{fig:lp3}
\end{figure}
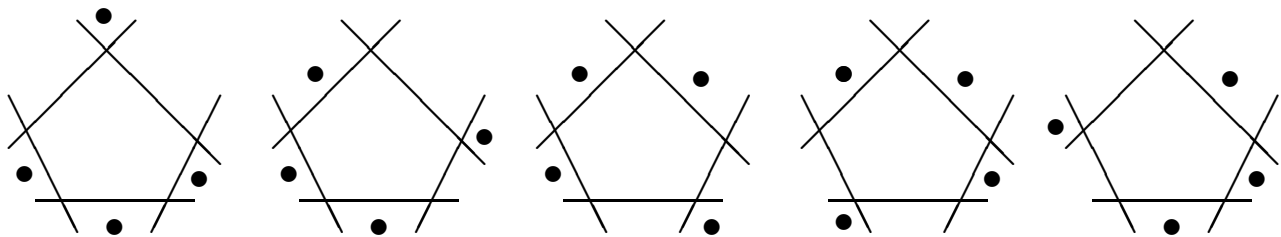

Let $S = \Z[\text{frozen variables}]$ be the polynomial ring generated by the variables associated to the non-internal regions.  Let $\FF = \Q(\text{all variables})$ be the rational function field in all the variables associated to internal or non-internal regions.
Let $t_{\W}$ be the initial seed with cluster variables corresponding to internal regions and exchange polynomials given by the rule above. Let $\A_{\W}$ be the LP algebra $t_{\W}$ generates.

\begin{example}
Consider the wiring diagram $\W$ shown in Figure \ref{fig:lp4}.  We have $$S = \Z[P,T,U,V,W,X,Y,Z]$$ and $\FF = \Frac(S[a,b,c])$.
 \begin{figure}[h!]
    \begin{center}
    \input{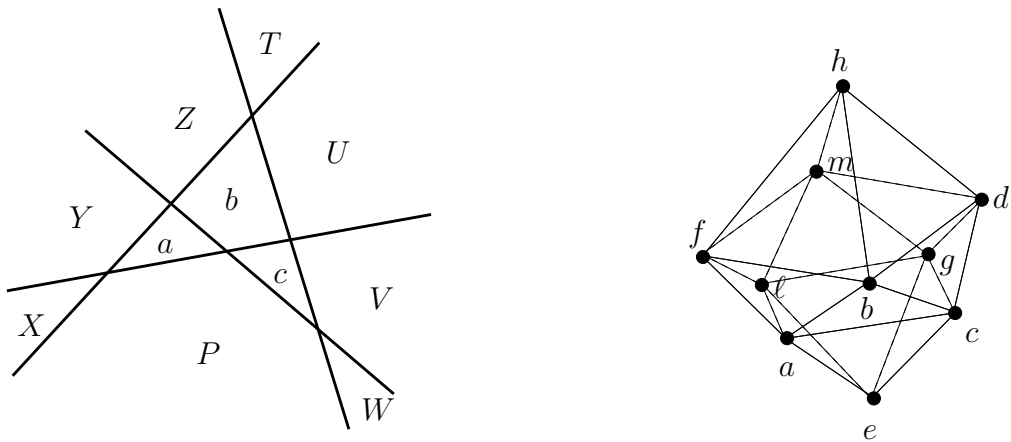}
    \end{center}
    \caption{A wiring diagram $\W$ and the variable labels of its regions; the cluster complex of the corresponding normalized finite type LP algebra $\A_{\W}$.}
    \label{fig:lp4}
\end{figure}
The initial seed in this case is 
$$t_{\W}=\{(a,b X+c Y+P Z),(b,a c T+c U Y+P U Z+a V Z),(c,P U+a V+b W)\}.$$
The exchange relations for $a$ and $c$ are
$$
ad = b X+c Y+P Z \qquad \text{and} \qquad cf=P U+a V+b W.
$$
These relations are instances of the cube recurrence  \cite{Pro}, which fit into a normalized finite type LP algebra.  Another seed in this LP algebra is
$$\{(a,e+U X),(e,a c T+c U Y+P U Z+a V Z),(c,e+W Z)\}$$
and each of the 16 seeds looks like either this seed, or the initial seed.

As we shall show in \cite{LP3}, our general LP algebra seed mutation is compatible with the combinatorics of wiring diagrams.  Namely, performing a braid move on the wiring diagram $\W$ corresponds to a LP algebra seed mutation at a cluster variable labeling a triangular region.  In particular, wiring diagrams $\W_1$ and $\W_2$ connected by braid moves give rise to seeds $t_{\W_1}$ and $t_{\W_2}$ that belong to the same LP algebra.

It turns out the dynamics of the LP algebras $\A_{\W}$ corresponds to the dynamics of transitions between factorizations of elements of {\it {electrical Lie groups}}, as defined in \cite{LPEL}. 
In fact, the LP algebras arising in this way are a natural analogue of cluster algebras appearing in double Bruhat cells of classical groups, as studied in \cite{CA3}.
The detailed study of these LP algebras and their relation with electrical Lie groups is the subject of the forthcoming \cite{LP3}.

\begin{remark}
The dynamics of cluster-like exchanges given by the formula in Figure \ref{fig:lp3} was studied by Henriques and Speyer \cite{HS} under the name of the 
{\it {multidimensional cube recurrence}}. Their work however deals only with the ``Pl\"ucker'' part of these algebras, corresponding to wiring diagrams.  In other words, the multidimensional cube recurrence only allows mutations corresponding to triangular bounded regions in the wiring diagram.  In fact, Henriques and Speyer state the problem of mutating beyond wiring diagrams as an open question. 
The Laurent phenomenon algebras $\A_{\W}$ constructed in this section accomplish this: in \cite{LP3} we show that these algebras contain all the seeds and mutations studied by Henriques and Speyer, but in addition one can perform mutations for regions with arbitrarily many sides.  Furthermore,  since LP algebras can be mutated indefinitely in all directions, we can keep mutating even after 
we have left the part described by wiring diagrams. 
\end{remark}
\end{example}

\end{document}